\documentclass[a4paper,reqno]{amsart}
 \pdfoutput=1
\usepackage[utf8]{inputenc}
\usepackage[T1]{fontenc}
\usepackage{lmodern}
\usepackage[final]{microtype}
\usepackage{amsmath}
\usepackage{amsthm}
\usepackage{amsfonts,amssymb}
\usepackage[english]{babel}
\usepackage{tikz}
\usepackage[backref=true,backend=bibtex]{biblatex}
\usepackage{tikz-cd}
\makeatletter
\newlength{\faktor@zaehlerhoehe}
\newlength{\faktor@nennerhoehe}
\DeclareRobustCommand*{\faktor}[3][]
{
   { \mathpalette{\faktor@impl@}{{#1}{#2}{#3}} }
}
\newcommand*{\faktor@impl@}[2]{\faktor@impl#1#2}
\newcommand*{\faktor@impl}[4]{
   \settoheight{\faktor@zaehlerhoehe}{\ensuremath{#1#2{#3}}}   \settoheight{\faktor@nennerhoehe}{\ensuremath{#1#2{#4}}}      \raisebox{0.5\faktor@zaehlerhoehe}{\ensuremath{#1#2{#3}}}      \mkern-5mu\diagup\mkern-4mu      \raisebox{-0.5\faktor@nennerhoehe}{\ensuremath{#1#2{#4}}}}
\makeatother
\usepackage{csquotes}
\usepackage{aliascnt}
\usepackage[hidelinks, pdfstartview={FitV},bookmarksopen=true]{hyperref}

\makeatletter\AtBeginDocument{\hypersetup{pdftitle = {\@title}, pdfauthor = {\@author} }}\makeatother

\usetikzlibrary{matrix,arrows,chains,calc,fadings,decorations.pathreplacing}
\tikzset{commutative diagrams/arrow style=math font}

\title[Equivariant Characteristic Forms]{Equivariant Characteristic Forms in the Cartan model and Borel equivariant Cohomology}

\author{Andreas Kübel}
\address{A.K., Max Planck Institute for Mathematics in the Sciences, Inselstraße 22, 04103 Leipzig, Germany}
\email{kuebel@mis.mpg.de}

\author{Andreas Thom}
\address{A.T., Institut f\"ur Geometrie, TU Dresden, 01062 Dresden, Germany}
\email{andreas.thom@tu-dresden.de}

\date{November 10, 2015}
\newcommand{\emphind}[1]{\emph{#1}}
\newcommand{\symindex}[2][a]{}
\AtBeginDocument{}

\theoremstyle{plain}\makeatletter
\g@addto@macro{\thm@space@setup}{\thm@headpunct{}}
\makeatother
\newtheorem{theorem}{Theorem}[section]
\def\Mynewtheorem#1#2{\newaliascnt{#1}{theorem}
\newtheorem{#1}[#1]{#2}
\aliascntresetthe{#1}
\expandafter\def\csname #1autorefname\endcsname{#2}
}
\Mynewtheorem{lemma}{Lemma}
\Mynewtheorem{prop}{Proposition}
\Mynewtheorem{defn2}{Definition + Proposition}
\Mynewtheorem{cor}{Corollary}
\theoremstyle{definition}
\Mynewtheorem{example}{Example}
\Mynewtheorem{remark}{Remark}
\Mynewtheorem{defn}{Definition}

\DeclareMathOperator{\id}{id}

\DeclareMathOperator{\End}{End}
\DeclareMathOperator{\Hom}{Hom}

\DeclareMathOperator{\Gl}{Gl}
\DeclareMathOperator{\Ad}{Ad}
\DeclareMathOperator{\pr}{pr}
\DeclareMathOperator{\sgn}{sgn}

\newcommand{\Cech}{\v{C}ech{}}
\newcommand{\from}{\colon}
\newcommand{\dual}{^\vee}
\newcommand{\defeq}{\mathrel{\mathop:}=}
\newcommand{\eqdef}{=\mathrel{\mathop:}}  
\newcommand{\del}{\partial}
\newcommand{\Natural}{\mathbb{N}}
\newcommand{\Real}{\mathbb{R}}

\newcommand{\C}{\mathbb{C}}

\newcommand{\g}{\mathfrak{g}}

\newcommand{\invers}{^{-1}}

\newcommand{\cl}{\mathrm{cl}}
\newcommand{\ddt}{\left.\frac{d}{dt}\right|_{t=0}}

\tikzset{  >=latex,   inner sep=0pt,  outer sep=2pt,  mark coordinate/.style={inner sep=0pt,outer sep=0pt,minimum size=3pt,
    fill=black,circle}}

\allowdisplaybreaks[1]
\begin{document}
\begin{abstract}
We show the compatibility of the differential geometric and the topological construction of equivariant characteristic classes for compact Lie groups. Our analysis motivates a differential geometric construction for equivariant characteristic classes in the non-compact case. 
    
This compatibility is generally assumed and used in various cases, but there is only a proof for compact connected Lie groups in the literature, see the work of Raoul Bott and Loring Tu. Our proof applies and generalizes ideas of Johan Dupont and Ezra Getzler.
\end{abstract}

\maketitle

\tableofcontents

\makeatletter{}\section{Introduction}\enlargethispage{\baselineskip}

The study of characteristic classes of vector bundles over manifolds has a long history. It is well-known that there are essentially two different ways to construct characteristic classes. One is differential geometric, using the curvature tensor of some chosen connection, the other one is topological, relying on a computation of the cohomology of the appropriate classifying space. It is also well-known that both approaches are equivalent, one giving slightly more information because the outcome is actually a differential form, the other giving slightly more information because the outcome is actually an integral cohomology class rather than a class in de Rham cohomology. A combination of both approaches leads to characteristic classes taking values in differential cohomology.

In this note, we will present a self-contained description of the two approaches in an equivariant setting and show the equivalence of the two approaches in case of equivariance with respect to the action of a compact group -- in complete analogy to the non-equivariant situation. The results that we are able to prove must be well-known to experts in the field, nevertheless only partial results have been published. Bott and Tu \cite{BottTu01} cover the case of equivariance with respect to the action of a connected Lie group. We do not claim any originality for the results presented in this paper -- the purpose of writing it is to provide a self-contained exposition of the constructions and complete proofs of the results. However, as a by-product our proofs are different from the original approaches and good for generalizations also to the non-compact setting -- see the remarks below. Moreover, we will make use of various explicit formulas that we obtain in order to study equivariant differential cohomology, see \cite{Ich2}. In fact, providing a sound basis for the development of equivariant differential cohomology was one of our main motivations to review the theory of equivariant characteristic classes.

\vspace{0.1cm}

Let $G$ and $K$ be Lie groups and let $M$ be a manifold. A $G$-equivariant (smooth) principal $K$-bundle is a (smooth) principal $K$-bundle $\pi\from E\to M$,  where $G$ acts from the left on $E$ and $M$, such that
\begin{enumerate}
    \item $\pi$ is $G$-equivariant, i.e. $\pi(gx)=g\pi(x)$ for any $g\in G$ and $x\in E$, and
    \item the left action of $G$ and the right action of $K$ commute, i.e. $(gx)k=g(xk)$ for any $g\in G$, $k\in K$ and $x\in E$.
\end{enumerate}

A $G$-\emphind{equivariant characteristic class} $c$ for $G$-equivariant principal $K$-bundles associates to every isomorphism class of topological $G$-equivariant principal $K$-bundles $\pi\from E\to M$ a cohomology class $c(E)\in H^*_G(X)$, such that $c((\bar f,f)^*E)=f^*c(E)$ for every pullback diagram \[\begin{tikzcd}
E'\arrow{r}{\bar f}\arrow{d}{\pi'}&E\arrow{d}{\pi}\\
M'\arrow{r}{f}&M
\end{tikzcd}\]
of $G$-equivariant principal $K$-bundles.

One way to define equivariant characteristic classes, is to apply the Borel construction (compare, e.g., \cite[Section 5.4]{Libine},\cite{BottTu01}): Let $E\to M$ be a $G$-equivariant principal $K$-bundle and $EG\to BG$ a universal principal $G$-bundle, then the quotient of the diagonal $G$-action on the Cartesian product
\[\pi_G=id\times \pi:EG\times_G E\to EG\times_G M\eqdef M_G\]
is a principal $K$-bundle over the homotopy quotient of the action on the base space and hence
\[c^G(E)\defeq c(EG\times_G E)\in H^*(EG\times_G M)=H^*_G(M)\]
defines a $G$-equivariant characteristic class.

For non-equivariant smooth principal $K$-bundles $E\to M$, there is the well-known Chern-Weil construction, which defines characteristic classes by differential geometric data: Given a symmetric polynomial $P$ on the Lie algebra of $K$, which is invariant under the adjoint action of $K$, then, for any connection on the bundle, the evaluation of $P$ on the curvature of the connection is a closed differential form on $E$, which is the pullback of a form on the base space, known as characteristic form. The equivalence class of the characteristic form in de Rham cohomology does not the depend on the connection and is a characteristic class of the bundle.

Nicole Berline and Michèle Vergne \cite{BerlineVergne} generalized this definition to the equivariant setting by replacing the curvature by the sum of the curvature and the moment map (see \autoref{Def:momentmap}), where now the equivariant characteristic form is an element in the Cartan model (see \autoref{sec:Cartan}). This is also discussed in \cite[pp.204]{BerlineGetzlerVergne} and \cite{Libine}.

Thus, starting from a invariant symmetric polynomial, one obtains an equivariant characteristic class in two ways: 1) Take the characteristic class associated to the polynomial by the Chern-Weil construction and apply it to the Borel construction or 2) use the definition of Berline and Vergne. We will show that both ways lead to the same class in equivariant cohomology. This is generally assumed the hold, see, e.g., \cite[311]{Jeffrey}, however without giving or citing a proof.

For compact connected Lie groups acting on the bundle, there is a proof of this compatibility given in \cite{BottTu01}. We will use different methods and give a proof that works for general compact Lie groups. Since the Cartan model does only compute equivariant cohomology for compact groups, this is the complete answer. For non-compact groups, nevertheless, there is an injection of the Cartan model to a model of equivariant cohomology defined by Ezra Getzler in \cite{Getzler}, and, in this sense, we also have compatibility for non-compact Lie groups.

Our proof will be based on simplicial manifolds and is a refinement of a construction of \cite{Dupont}. One can read parts of it as a proof of \cite[Theorem 3.1.1]{Getzler}, which Getzler claims to be proven in \cite{Dupont76}, where only a weaker statement is shown. For the relation between the de Rham cohomology on simplicial manifolds and the Cartan model, we will use arguments of \cite{Getzler}.

This work arose from the first author's PhD-thesis \cite{Diss}.
\section{Equivariant cohomology}
\subsection{Borel equivariant cohomology}
Let $M$ be a smooth manifold acted on from the left by a Lie group $G$. To define equivariant cohomology one uses two properties which one expects from such a theory: it should be homotopy invariant and for free actions, the equivariant cohomology should be the cohomology of the quotient. Recall that the total space of the classifying bundle, $EG$, is a contractible topological space with free $G$-action. Hence $EG\times M$ has the homotopy type of $M$ and the diagonal action is free. Hence one defines 
\[H^*_G(M)\defeq H^*(EG\times_G M),\] 
where $EG\times_G M$ \symindex[e]{EG\times_G M} is the quotient of $EG\times M$ by the diagonal action.
We are interested in differential form models for equivariant cohomology, but in general $EG$ is not a finite-dimensional manifold, hence we cannot use the usual de Rham cohomology. But there is a model for $EG$, which consist of finite dimensional manifold:

\subsection{Simplicial manifolds}
\label{sec:simpman}
The model of $EG\times_G M$ we are going to use is given by a simplicial manifold. To define this, recall (for more details see, e.g., \cite[Section I.1]{simphomtheory}) that the simplex category $\Delta$\symindex[d]{\Delta} is the category, whose objects are non empty finite ordered sets
\[[p]=\{0<1<\dots<p\},\text{ for } p\in\Natural,\] and the set of morphisms $\Delta([p],[q]), p,q\in\Natural$ is the set of order preserving maps $f\from [p]\to [q]$.

The morphisms of this category are generated by \emphind{coface maps} $\del^i\from [p-1] \to [p]$, where $\del^i$ is the map that misses $i$ and \emphind{codegeneracy maps} $\sigma^i\from [p]\to [p-1]$, where $\sigma^i$ is uniquely determined by hitting $i$ twice.
These maps satisfy the following \emph{cosimplicial relations}:
\begin{alignat*}{2}
\del^j\del^i&=\del^i\del^{j-1}&&\text{if } i<j\\
\sigma^j\del^i&=\del^i\sigma^{j-1}&&\text{if } i<j\\
\sigma^i\del^i&=\id=\sigma^i\del^{i+1}&&\\
\sigma^j\del^i&=\del^{i-1}\sigma^{j}&&\text{if } i>j+1\\
\sigma^j\sigma^i&=\sigma^i\sigma^{j+1}&&\text{if } i\leq j\\    
\end{alignat*}\symindex[d]{\del^i}\symindex[s]{\sigma^i}
\begin{defn}[see, e.g., {\cite[p.89]{Dupont}}]
    A \emph{simplicial manifold}\index{simplicial manifold} is contra-variant functor from the simplex category $\Delta$ to the category of smooth manifolds.
\end{defn}
Explicitly this is an $\Natural$-indexed family of manifolds with smooth \emph{face}\index{face maps} and \emphind{degeneracy maps} satisfying the simplicial relations, i.e.
\begin{align*}
\del_i\circ\del_j&=\del_{j-1}\circ\del_i, \text{ if } i<j\\
\sigma_i\circ \sigma_j&=\sigma_{j+1}\circ \sigma_i, \text{ if } i\leq j\\
\del_i\circ \sigma_j &= \begin{cases}
                    \sigma_{j-1}\circ\del_i , &\text{ if } i<j\\
		    \id , &\text{ if } i=j,j+1\\
		    \sigma_{j}\circ\del_{i-1} , &\text{ if } i>j+1
                   \end{cases}
\end{align*}\symindex[d]{\del_i}\symindex[s]{\sigma_i}
\begin{example}\label{ex:actionmanifold}
Our most important example of a simplicial manifold is the following (compare \cite[p.316]{Gomi},\cite[section 3.2]{Getzler}): Let the Lie group $G$ act smoothly on the manifold $M$. We define the simplicial manifold \symindex[g]{G^\bullet M} $G^\bullet\times M$ by
\[G^\bullet\times M = \{G^p\times M\}_{p\geq 0},\]
where $G^p$ stands for the $p$-fold Cartesian product of $G$. The face maps $G^p\times M\to G^{p-1}\times M$ are given as
\begin{align*}
\del_0(g_1, \dots , g_p, x) &= (g_2, \dots, g_p, x)\\
\del_i(g_1, \dots , g_p, x) &= (g_1, \dots , g_{i-1}, g_{i}g_{i+1},\dots , g_p, x)\text{ for } 1 \leq i \leq p - 1\\
\del_p(g_1, \dots , g_p, x) &= (g_1, \dots , g_{p-1}, g_p x)
\end{align*}
and the degeneracy maps for $i=0,\dots,p$ by
\begin{align*}
 \sigma_i: G^p\times M&\to G^{p+1}\times M\\
(g_1,\dots,g_p,x)&\mapsto (g_1,\dots,g_i,e,g_{i+1},\dots,g_p,x).
\end{align*}
These maps satisfy the simplicial relations.
\label{rm:action}
In particular, for $p=1$ the map $\del_1$ equals the group action, while $\del_0$ is the projection onto the second factor, i.e. onto $M$. 
\end{example}

To turn a simplicial manifold into a topological space, recall that the standard $n$-simplex is defined as
\[\Delta^n=\left\{t=(t_0,\dots,t_n)\in\Real^{n+1}\middle|t_i\geq0 \text{ for all } i, \sum_i t_i=1\right\}.\]
\begin{defn}[see, e.g., {\cite[p.75]{Dupont}}]
    The (fat) \emphind{geometric realization} of a simplicial manifold $M_\bullet$, is the topological space
 \[\|M_\bullet\|=\bigcup_{p\in\Natural} \Delta^p\times M_p/\sim\]
with the identifications
\[(\del^it,x)\sim(t,\del_i x)\text{ for any }x\in M_p,\, t\in\Delta^{p-1}, i=0,\dots,n \text{ and } p=1,2,\dots .\]
\end{defn}
\begin{example}
The geometric realization of the simplicial manifold $G^\bullet\times M$ is a model of $EG\times_G M$ and, in particular, if $M$ is single point the geometric realization of $G^\bullet\times pt$ is a model of the classifying space $BG$ (compare \cite[pp.75]{Dupont}).
\end{example}

\subsection{Simplicial de Rham cohomology}
As shown by Dupont in \cite[Section 6]{Dupont}, the cohomology of the geometric realization of a simplicial manifold equals the cohomology of the double complex of simplicial differential forms: Namely, let $M_\bullet=\{M_p\}$ be a simplicial manifold. For any $p$, complex valued differential forms\footnote{In the whole article differential forms, polynomials and coefficients will be complex for simplicity. All arguments hold for real numbers, too.} on $M_p$ form, together with the exterior derivative, the cochain complex $(\Omega^*(M_p),d)$. The face and degeneracy maps of $M_\bullet$ induce, via pullback, face and degeneracy maps between the differential forms on $M_p$ and $M_{p\pm 1}$. 

Thus on the bigraded collection of vector spaces
\[\Omega^{p,q}(M)=\Omega^q(M_p),\]
there is a horizontal differential $d:\Omega^{p,q}(M_\bullet)\to \Omega^{p,q+1}(M_\bullet)$, given by the exterior differential and vertical differential 
\[\del: \Omega^{p,q}(M_\bullet)\to \Omega^{p+1,q}(M_\bullet),\]
given by the alternating sum of pullbacks along the face maps
\begin{equation}\label{del}
\del(\omega)=\sum_{i=0}^{p+1} (-1)^i\del_i^*\omega.
\end{equation}
\begin{prop}  $(\Omega^{p,q}(M_\bullet),d,(-1)^q\del)_{p,q}$ is a double complex.
\end{prop}
\begin{proof}
 $(d+(-1)^q\del)^2=0$, as the exterior derivative squares to zero, i.e. $d^2=0$ , $\del^2=0$ by the simplicial relations and $d\del=\del d$ as $d$ is functorial.
\end{proof}

Recall, that the cohomology of a double complex is defined to be the cohomology of its total complex, i.e., the cochain complex
\[\left(\bigoplus_{p+q=n}\Omega^{p,q}(M_\bullet),d+(-1)^q\del\right)_{n\in\Natural}.\]
In particular, for the simplicial manifold $G^\bullet\times M$, we have the double complex $\Omega^q(G^p\times M)$, what is a first de Rham type model for equivariant cohomology by the following Proposition.

\begin{prop}[Prop. 6.1 of {\cite{Dupont}}]\label{thm:realiso}
    Let $M_\bullet$ be a simplicial manifold. There is a natural isomorphism \[H^*(\Omega^{\bullet,*}(M_\bullet),d+(-1)^*\del)\cong H^*(\|M_\bullet\|,\C),\]
    between the cohomology of the de Rham double complex of simplicial differential forms and the cohomology of the geometric realization. 
\end{prop}
\begin{proof}
    The detailed proof can be found in \cite{Dupont}. We will give a sketch for completeness: Introduce the simplicial singular double complex $S^{p,q}(M_\bullet)$. Therefore, let, for some manifold $M$, $S_q(M)$ denote the set of all smooth singular $q$-simplices in $M$, i.e.  smooth maps from the standard $q$-simplex $\Delta^q$ to $M$. The maps from the set $S_q(M)$ to $\C$ form a vector space, denoted by $S^q(M)$, and the alternating sum over the restrictions to the faces of $\Delta^q$ yields a coboundary map $\delta$ on $S^q(M)$. Define $S^{p,q}(M_\bullet)\defeq S^q(M_p)$ with horizontal boundary map $\delta$. Since the face maps of the simplicial manifold are smooth, they turn a smooth singular simplex in $M_p$ into a smooth singular simplex in $M_{p-1}$. Thus yield, by linear extension, a homomorphism $S^q(M_{p-1})\to S^q(M_p)$. The alternating sum over this homomorphisms is the vertical coboundary map $\del$.
    
    Now check that 
    \begin{enumerate}
        \item The geometric realization of the bisimplicial set $S(M)=S_q(M_p)$ is homeomorphic to the realization of the simplicial space $p\mapsto\|q\mapsto S_q(M_p)\|$;
        \item $\|q\mapsto S_q(M_p)\|$ is homotopy equivalent to $M_p$;
        \item the total complex $\bigoplus_{p+q=n} S^{p,q}(M_\bullet)$ with coboundary map $\delta+(-1)^q\del$ is the singular chain complex of $\|S(M)\|$.
    \end{enumerate}
	This implies that $H^*(S^{p,q}(M_\bullet),\delta,(-1)^q\del)=H^*(\|M_\bullet\|,\C)$.
	
	On the other hand, integration of the pullback form over the simplex, induces a map 
	$\Omega^{p,q}(M_\bullet)\to S^{p,q}(M_\bullet)$,
	which is, as easily checked, a map of double complexes. Moreover, as in the non-simplicial situation, it induces an isomorphism in cohomology. 
\end{proof}

\subsection{The Cartan model}\label{sec:Cartan}
A well-known de Rham-like model for equivariant cohomology goes back to Henri Cartan (\cite{Cartan}). Our Exposition follows \cite{Libine}. 
Let $G$ be a compact Lie group acting smoothly on the smooth manifold $M$ and denote the Lie algebra of $G$ by $\g=T_eG$. Let $S^*(\g\dual)$ be the symmetric tensor algebra of the (complex) dual of the Lie algebra $\g\dual$. The group $G$ acts on this algebra by the coadjoint action and on $\Omega^*(M)$ by pulling back forms along the action map. Hence we have a $G$-action on $S^*(\g\dual)\otimes\Omega^*(M)$. The invariant part of this algebra, $(S^*(\g\dual)\otimes\Omega^*(M))^G$, is what one calls the \emphind{Cartan complex} and is denoted by $\Omega_G^*(M)$\symindex[o]{\Omega_G^*(M)}. In other words: The Cartan complex consists of $G$-equivariant polynomial maps $\omega\from \g\to \Omega^*(M)$. Let $\omega_1,\omega_2\in \Omega_G^*(M)$, then there is a wedge product 
\[(\omega_1\wedge \omega_2)(X)=\omega_1(X)\wedge \omega_2(X).\]
On this algebra a differential is defined by
\[d_C \omega (X)=d(\omega(X))+\iota(X^\sharp)\omega(X),\]
for $\omega\in \Omega_G^*(M)$ and $X\in \g$,
i.e., by the sum of the exterior differential on the manifold and the contraction with the fundamental vector field: For any $X\in\g$ and $m\in M$, define $X^\sharp_m=\ddt (e^{tX}\cdot m)\in T_mM$. To make this differential raise the degree by one, the grading on $\Omega_G^*(M)$ is given by
\[\text{twice the polynomial degree} + \text{the differential form degree}.\]
\begin{lemma}
    $(\Omega^*_G,d_C)$ is a cochain complex.
\end{lemma}
\begin{proof}
First, observe that $d_C$ increases the total degree by one, since $d$ increases the differential form degree, and the contraction $\iota$, while decreasing the form degree by one, increases the polynomial degree by one. Next, one has to check, that the differential really maps invariant forms to invariant forms and that it squares to zero.

Let $\omega\in \Omega_G^*(M))$ and $X\in \g$. 
\begin{align*}
d_C \omega(\Ad_g X)&=d(\omega(\Ad_g X))+\iota((\Ad_g X)^\sharp)\omega(\Ad_g X)\\
&=d(g \omega(X))+\iota(g X^\sharp g\invers) g(\omega(X))\\
&=g d(\omega(X))+g\iota( X^\sharp ) g\invers g(\omega(X))\\&=gd_C \omega 
\end{align*}
Thus $d_C \omega$ is $G$-equivariant. Moreover, we have
\[d_C^2\omega(X)= d^2\omega(X)+d\iota(X^\sharp )\omega(X)+\iota(X^\sharp )d\omega(X)+\iota(X^\sharp )^2\omega(X)=L_{X^\sharp}\omega(X),\]
by the (Élie) Cartan formula and
\[L_{X^\sharp}\omega(X)=\ddt\hspace{-8pt}\exp(tX)\omega(X)=\ddt\hspace{-8pt} \omega(\exp(-tX)X\exp(tX))=\ddt \hspace{-8pt}\omega(X)=0.\]
Thus $d_C$ squares to zero, i.e., it is a boundary operator.
\end{proof}

In the special case of $M=pt$, i.e., of a single point, the Cartan algebra reduces to the algebra of \emph{invariant symmetric polynomials} \index{invariant symmetric polynomials}\symindex[i]{I^k(G)}
\[I^k(G)=((S^*(\g\dual)\otimes\Omega^*(pt))^G)^k=(S^k(\g\dual))^G.\]

\subsection{Getzlers resolution}\label{sec:Getzler}
In order to investigate cohomology of actions of non-compact groups, Ezra Getzler \cite[Section 2]{Getzler} defines a bar-type resolution of the Cartan complex. We will apply his ideas slightly different: The complex defined by Getzler will allow us to compare the de Rham model on the simplicial manifold $G^\bullet\times M$ with the Cartan model.

Let, as before, a Lie group $G$ act on a smooth manifold $M$ from the left. Define $\C$-vector spaces $C^p(G,S^*(\g\dual)\otimes \Omega^*(M))$ consisting of smooth maps, from the $p$-fold Cartesian product
\[G^p\to S^*(\g\dual)\otimes \Omega^*(M),\]
to the space of polynomial maps from $\g$ to differential forms on $M$. These groups come with three gradings: The differential form degree on $M$, the polynomial degree and the number $p$ of copies of $G$, which we call simplicial degree of this complex. The Cartan boundary operator $d+\iota$ now induces a map $(-1)^p(d+\iota)$ which now, as we are not restricted to the $G$-invariant part of $S^*(\g\dual)\otimes \Omega^*(M)$, will not square to zero, but \[\left((-1)^p(d+\iota)\right)^2=d\iota+\iota d=L\] is the Lie derivative, by the Cartan formula. The `simplicial' degree $p$ is increased by the `simplicial' boundary map
\[\bar d\from C^k(G,S^*(\g\dual)\otimes \Omega^*(M))\to C^{k+1}(G,S^*(\g\dual)\otimes \Omega^*(M))\]
defined by
\begin{multline*}
 (\bar d f)(g_0,\dots,g_k|X)\defeq f(g_1,\dots,g_k|X)+ \sum_{i=1}^k (-1)^i f(g_0,\dots,g_{i-1}g_i,\dots,g_k|X)\\+(-1)^{k+1}g_kf(g_0,\dots,g_{k-1}|\Ad(g_k\invers)X)
\end{multline*}
for $g_0,\dots,g_k\in G$ and $X\in \g$.

Note, in particular, that the kernel of
\[\bar d\from C^0(G,S^*(\g\dual)\otimes \Omega^*(M))\to C^1(G,S^*(\g\dual)\otimes \Omega^*(M))\] 
is exactly $\Omega_G^*(M)$. Moreover, in case of a discrete Group $G$, $\g=0$ and thus one checks, that 
\[C^p(G,S^*(\g\dual)\otimes \Omega^*(M))=C^p(G,\Omega^*(M))= \Omega^{p,*}(G^\bullet\times M)\] and $\bar d$ is equal to $\del$.

In the case of a compact Lie group, the map $\bar d$ admits a contraction (compare, e.g., \cite[322]{Gomi}):
\begin{lemma}\label{lem:IntGroupGetzler} Integration over the group, with respect to a right invariant probability measure, defines a map 
\begin{align}
\int_G:C^p(G,S^*(\g\dual)\otimes \Omega^*(M))&\to C^{p-1}(G,S^*(\g\dual)\otimes \Omega^*(M)) \label{eq:IntGroupGetzler}\\
\left(\int_G f\right)(g_1,\dots,g_{p-1},m)&=(-1)^i\int_{g\in G} f(g,g_1,\dots,g_{p-1},m)dg\nonumber
\end{align}
such that $\bar d \int_G f=f$ if $\bar d f=0$.
\end{lemma}
\begin{proof}This is proven by a direct calculation, which makes use of the left invariance of the measure.
\end{proof}

Thus, for compact groups, the vertical cohomology of this bi-graded collection of groups is the Cartan complex. \\[2ex]
One can turn the triple-graded collection $C^p(G,S^*(\g\dual)\otimes \Omega^*(M))$ of groups into a double complex. Therefore Getzler defines another map,
\[\bar{\iota} \from C^p(G,S^l(\g\dual)\otimes \Omega^m(M))\to C^{p-1}(G,S^{l+1}(\g\dual)\otimes \Omega^m(M)),\]
given by the formula
\[(\bar{\iota} f)(g_1,\dots,g_{p-1}|X)\defeq \sum_{i=0}^{p-1}(-1)^i \ddt f(g_1,\dots,g_i,\exp(tX_i),g_{i+1},\dots,g_{p-1}|X),\]
where $X_i=\Ad(g_{i+1}\dots g_{p-1})X$. 
 \begin{lemma}[Lemma 2.1.1. of {\cite{Getzler}}]The map $\bar{\iota}$ has the following properties:
  \[\bar{\iota}^2=0 \text{ and }\bar d\bar\iota+ \bar\iota\bar d= -L.\]
 \end{lemma}
\begin{proof}
 This is shown in \cite{Getzler} by recollection of the sums in the definition of $\bar{\iota}$ and $\bar d$.
\end{proof}

Moreover one obtains:
\begin{lemma}[Corollary 2.1.2. of {\cite{Getzler}}]
 $d_G= \bar d + \bar \iota + (-1)^p(d + \iota)$ is a boundary operator on the total complex $\bigoplus_{p+2q+r=n} C^p(G,S^q(\g\dual)\otimes \Omega^r(M))$.
\end{lemma}
\begin{proof} $d_G$ increases the total index by one, as $\bar d$ increases the first index, $d$ increases the third index, $\iota$ decreases the third, while it is increasing the second index and $\bar\iota$ decreases the first index, while it is increasing the second one.

As $d$ and $\iota$ are equivariant under the $G$-action, they commute with $\bar d$. And as $d$ and $\iota$ only act on the manifold $M$ and not on the group part, the same is true for $\bar\iota$. Thus
 \begin{align*}
  d_G^2&=(\bar d + \bar\iota)^2 +(-1)^p(\bar d + \bar \iota)(d + \iota) +(-1)^{p\pm 1}(d + \iota)(\bar d + \bar \iota)+ (d + \iota)^2\\
&=\bar d\bar \iota+ \bar\iota \bar d + (d\iota+\iota d)\\
&=-L + L = 0.
 \end{align*}
\end{proof}
The proof implies that 
\[\left(\left(\bigoplus_{\substack{k+l=p\\l+m=q}} C^k(G,S^l(\g\dual)\otimes \Omega^m(M))\right)_{\!\!\!p,q}\hspace{-.5em},\bar d+(-1)^k\iota,(-1)^kd+\bar\iota\right)\] is a double complex.
\begin{remark}
    The reader, who compares this with the original paper of Getzler will note that we changed some signs. It just seems more natural to us in this way. Furthermore Getzler uses some reduced subcomplex, which is, by standard arguments on simplicial modules (compare Proposition 1.6.5 in \cite{Loday}), quasi-isomorphic to the full complex, which we have taken.
\end{remark}

\subsection{A quasi-isomorphism}\label{sec:quasiiso}
The map we discuss below, defined in \cite[Section 2.2.]{Getzler}, relates the double complex $C^*(G,S^*(\g\dual)\otimes \Omega^*(M))$ to the de Rham double complex $\Omega^*(G^\bullet\times M)$. Thus we have an explicit identifications of chains in the one complex with chains in the other complex. This will be of particular interest to us in the discussion of equivariant characteristic forms (Section \ref{sec:Dupontforms}).

\begin{defn}[Def. 2.2.1. of \cite{Getzler}] \label{def:Getzlermap} The map 
\[\mathcal J\from \Omega^q(G^p\times M)\to \bigoplus_{\substack{k+l=p\\l+m=q}} C^k(G,S^l(\g\dual)\otimes \Omega^m(M))\] \symindex[j]{\mathcal J} \index{Getzler's map} is defined by the formula
 \[\mathcal J(\omega)(g_1,\dots,g_k|X)\defeq\sum_{\pi\in S(k,p-k)} \sgn(\pi)\left(i_\pi\right)^*\left(\iota_{\pi(k+1)}(X^{(\pi)}_{k+1})\dots\iota_{\pi(p)}(X^{(\pi)}_p)\omega\right).\]
Here $S(k,p-k)$ is the set of shuffles, i.e., permutations $\pi$ of $\{1,\dots,p\}$, satisfying
\[\pi(1)<\dots<\pi(k) \text{ and } \pi(k+1)<\dots<\pi(p),\]
$X^{(\pi)}_j=\Ad(g_{m}\dots g_{k})X$, where $m$ is the least integers less than $k$, such that $\pi(j)<\pi(m)$,
$\iota_j$ means, that the Lie algebra element should be a tangent vector at the $j$-th copy of $G$,
and $i_\pi\from G^k\times M\to G^p\times M$ is the inclusion $x\mapsto (h_1,\dots,h_p,x)$ with \[h_j=\begin{cases}g_m&\text{if } j=\pi(m), 1\leq m\leq k\\ e\in G &\text{otherwise,}\end{cases}\] which is covered by the bundle inclusion $TM\to T(G^p\times M)$.
\end{defn}
Observe that the image of $\omega$ under $\mathcal J$ does only depend on the zero form part and, in direction of any copy of $G$, on the one form part at the identity $e\in G$.

The next Lemma -- which is mainly a citation of \cite[Lemma 2.2.2.]{Getzler}, but with signs corrected -- shows, that the map $\mathcal J$ can be interpreted as a map of double complexes.
\begin{lemma}The map $\mathcal J$ respects the boundaries with the correct sign, i.e.,
 \[\mathcal J \circ \del = \left(\bar d+(-1)^p\iota\right) \circ\mathcal J\] and, after decomposing $d=d_G+d_M$ with respect to the Cartesian product $G^p\times M$
 \[\mathcal J \circ \left(\left(-1\right)^p d_M \right)= (-1)^{p'} d \circ \mathcal J \text{ and }\mathcal J \circ \left(-1^p\right)d_G = \bar\iota \circ \mathcal J,\]
where $p$ is the simplicial degree before and $p'$ the simplicial degree after application of the map $\mathcal J$,
\end{lemma}
\begin{proof}
    The following four types of terms contribute to $\mathcal J \circ \del$:
\begin{enumerate}
    \item those terms where $\del$ acts by the group multiplication $G\times G\to G$ and $\mathcal J$ take a one form component on one of these groups: these parts cancel by symmetry;
\item	those parts where $\del$ acts by the group multiplication or the action on $M$ and $\mathcal J$ takes the zero form component on this part: these contribute to $\bar d \circ \mathcal J$;
\item those terms where $\del$ corresponds to the action of $G$ on $M$ and $\mathcal J$ takes the one form on this corresponding $G$ at $e$ yield $\iota\circ \mathcal J$. The sign comes from the fact that $\del_p$ has this sign in $\del$.
\end{enumerate}
 This proves the first equation. For the second decompose the exterior derivative $d=\sum_{k=1}^pd_{G^{(k)}}+d_M$ on $G^p\times M$ further into a part corresponding to each copy of $G$ and one corresponding to $M$. One checks immediately that $\mathcal J \circ \left(-1\right)^p d_M = (-1)^{p'}d \circ \mathcal J$, where the sign difference comes from interchanging $d_M$ with the contractions. For the last equation, which contains $\bar\iota$, note that one can restrict to $\omega\in \Omega^q(G^p\times M)$, whose degree on each copy of $G$ is either zero or one. Let $\pi\in S(l,p-l)$ be a shuffle and $\omega\in \Omega^q(G^p\times M)$ be a form, such that the differential form degree on the $\pi(k)$-th copy of $G$ is zero if $k\leq l$ and one if $k\geq l+1$. Then for any $X\in \mathfrak g$ we can calculate
\begin{align*}
  (\bar\iota& \circ \mathcal J \omega)(X)\\
&=\sum_{k=1}^l (-1)^k L_{X_k}^k  (\mathcal J \omega)(X)\\
\intertext{where $L^k$ should denote the Lie derivative on the $k$-th $G$,}
&=\sum_{k=1}^l (-1)^k \sgn(\pi)\sigma_{k-1}^*\iota_k(X_k)d_{G^{(k)}}  i_\pi^*\left(\iota_{\pi(l+1)}(X^{(\pi)}_{l+1})\dots\iota_{\pi(p)}(X^{(\pi)}_{p}) \omega\right)\\
&=\sum_{k=1}^l (-1)^k \sgn(\pi)\sigma_{k-1}^*\iota_k(X_k)i_\pi^*\left(d_{G^{(\pi(k))}}  \iota_{\pi(l+1)}(X^{(\pi)}_{l+1})\dots\iota_{\pi(p)}(X^{(\pi)}_{p}) \omega\right)\\
&=\sum_{k=1}^l (-1)^k \sgn(\pi)\sigma_{k-1}^*\iota_k(X_k)i_\pi^*  (-1)^{p-l}\left(\iota_{\pi(l+1)}(X^{(\pi)}_{l+1})\dots\iota_{\pi(p)}(X^{(\pi)}_{p}) d_{G^{(\pi(k))}}\omega\right)\\
&\quad + \text{terms which include a Lie derivative $L^k_{X_k}(X_{l+j})$ }\\
\intertext{The sign in the first term comes from the fact that the $d$ and $\iota$ anti-commute, since they act on different $G$'s. Moreover the other terms vanish, as each of them contains a factor of the form $\ddt e^{tX}Xe^{-tX}=[X,X]=0$.}
&=\sum_{k=1}^l (-1)^{p-l+k} \sgn(\pi)\sigma_{k-1}^*i_\pi^*\!\left(\iota_{\pi(k)}(X_k) \! \left(\iota_{\pi(l+1)}(X^{(\pi)}_{l+1})\dots\iota_{\pi(p)}(X^{(\pi)}_{p}) d_{G^{(\pi(k))}}\omega\right)\!\right)\\
\intertext{Now define $\pi_k\in S(l-1,p-l+1)$ as the shuffle, which is obtained from $\pi$ by transpose $k$ and $l$ and resorting the two groups.}
&=\sum_{k=1}^l (-1)^{p-l+k} \sgn(\pi)i_{\pi_k}^*\!\left(\iota_{\pi(k)}(X_k) \! \left(\iota_{\pi(l+1)}(X^{(\pi)}_{l+1})\dots\iota_{\pi(p)}(X^{(\pi)}_{p}) d_{G^{(\pi(k))}}\omega\right)\!\right)\\
\intertext{The sign of $\pi$ and $\pi_k$ differ by a $(-1)^{l-k}$ for transposing $\pi(k)$ into the second group and a sign change for every transposition which is necessary to reorder the second group. These reordering sign also occur a second time, when reordering the contractions. Thus they cancel out each other.}
&=\sum_{k=1}^l (-1)^{p} \sgn(\pi_k)i_{\pi_k}^* \left(\iota_{\pi_k(l)}(X^{(\pi_k)}_{l})\dots\iota_{\pi_k(p)}(X^{(\pi_k)}_{p}) d_{G^{(\pi(k))}}\omega\right)\\
&=(-1)^p\sum_{k=1}^l \mathcal J \left(d_{G^{(\pi(k))}}\omega\right)(X)\\
&=(-1)^p \mathcal J \left(\sum_{k=1}^ld_{G^{(\pi(k))}}\omega\right)(X)\\
&=(-1)^p \mathcal J \left(d_{G^p} \omega\right)(X).
\end{align*}
Note for the last step, that $\mathcal J$ vanishes on forms, whose degree on any copy of $G$ is larger than one.
\end{proof}

Further, the map $\mathcal J$ induces an isomorphism in the cohomology of the associated total complexes.
\begin{theorem}[Theorem 2.2.3. of \cite{Getzler}]
 $\mathcal J$ is a quasi-isomorphism.
\end{theorem}

\section{Equivariant characteristic classes}

In the introduction, we described, how characteristic classes induce equivariant characteristic classes in the Borel model. We will show, that any equivariant characteristic class is defined in this way. Recall, that any characteristic class for principal $K$-bundles $c$ yields to an equivariant characteristic class $c^G$ by the following procedure (compare, e.g., \cite[Section 5.4]{Libine},\cite{BottTu01}):
Let $E$ be a $G$-equivariant principal $K$-bundle and let $EG\to BG$ denote a universal $G$-bundle, i.e. $EG$ is a contractible topological space with free $G$-action and $BG$ is the quotient of this action. One defines 
\[c^G(E)\defeq c(EG\times_G E)\in H^*(EG\times_G X)=H^*_G(X).\]
Note that by 
\[\pi_G=id\times \pi:EG\times_G E\to EG\times_G M\eqdef M_G\]
one constructs a principal $K$-bundle from the $G$-equivariant principal $K$-bundle $E\to M$.
\begin{lemma}
    The association procedure $c\mapsto c^G$ is an one to one correspondence between characteristic classes and equivariant characteristic classes.
\end{lemma}
\begin{proof}
    Any principal $K$-bundle $\pi\from E\to X$ can be understood as a $G$-equivariant principal $K$-bundle with trivial $G$-action. This holds for morphisms, too. Moreover, any section $f\from X\to \to BG\times X$ induces a map $f^*\from H^*_G(X)=H^*(EG\times_G X)=H^*(BG\times X)\to H^*(X)$. Thus an equivariant characteristic class naturally yields a characteristic class. We are now going to prove that this is an inverse to $c\mapsto c^G$.

Let $c$ be a characteristic class and $E$ be a principal $K$-bundle. Then $c^G(E)=c(EG\times_G E)=c(BG\times E)$. Let $\pr\from BG\times X\to X$ denote the projection, then $BG\times E=\pr^*E$ and $f^*c(BG\times E)=f^*c(\pr^* E)=f^*(\pr^*c( E))=(\pr\circ f)^*c( E)=c(E)$, since $\pr\circ f=\id_X$.

On the other hand, let $c$ be $G$-equivariant characteristic class. We have to show that for any $G$-equivariant principal $K$-bundle $F\to Y$
\[c(F)=f^*c(EG\times_G F)\in H^*_G(Y)=H^*(EG\times_GY),\]
where $f\from EG\times_GY\to BG\times(EG\times_GY)$ is an inclusion as above. Both squares in the commutative diagram
\[\begin{tikzcd}[column sep=large]
    EG\times_GF\arrow{d}&EG\times F\arrow{d}\arrow{l}\arrow{r}&F\arrow{d}\\
EG\times_GY&EG\times Y\arrow{l}[above]{\pr_1}\arrow{r}{\pr_2}&Y
\end{tikzcd}\]
are pullbacks of $G$-equivariant bundles, hence $\pr_2^*c(F)=\pr_1^*c(EG\times_G F)$ in $H^*_G(EG\times Y)$. Since $EG$ is contractible and Borel cohomology is invariant under (non-equivariant) contractions, $\pr_2^*$ is an isomorphism. Thus we only have to show that $(\pr_2^*)\invers \circ\pr_1^*=f^*$. This follows, because the diagram 
\[\begin{tikzcd}[column sep=tiny]
      EG\times_G Y\arrow{rr}{f}&&BG\times(EG\times_G Y)\\
	&(EG\times EG\times Y)/G\arrow{ul}{\pr_2}\arrow{ur}[below]{\pr_1}&
  \end{tikzcd}\]
commutes up to homotopy, since $EG$ is contractible.
\end{proof}
\begin{remark}
    This lemma is a reformulation of the basic statement of \cite{May85} and the reason why  Peter May proposed, that equivariant characteristic classes in Borel cohomology are too `crude'.
\end{remark}

\section{Classifying bundle and classifying map}

\subsection{Simplicial bundles}
Let $K$ be a Lie group. A simplicial manifold, whose realization is the universal principal $K$-bundle $EK\to BK$ is given in the following way: $NK\defeq K^\bullet\times pt$ is the simplicial manifold of the action of $K$ on a point (see Example \ref{ex:actionmanifold}) and $\left(N\overline K\right)_\bullet$ is  \symindex[N]{NK}\symindex[N]{N\overline K} given as $N\overline K_p=K^{p+1}$ with face maps 
\[\del_i(k_0,\dots,k_p)=(k_0,\dots,\hat k_i,\dots,k_p),\]
where the hat indicates that this element is left out, and degeneracy maps 
\[\sigma_i(k_0,\dots,k_p)=(k_0,\dots,k_i,k_i,\dots,k_p).\]
The $K$-action on $N\overline K$, given by 
\[(k_0,\dots,k_p)k=(k_0k,\dots,k_pk),\]
is compatible with the face and degeneracy maps and hence a simplicial action. Moreover, there is a simplicial map
\[\gamma:N\overline K\to NK, (k_0,\dots,k_p)\mapsto (k_0k_1^{-1},\dots,k_{p-1}k_p^{-1}).\]

\begin{lemma}[Prop. 5.3 of {\cite{Dupont}}]
 $\|\gamma\|\colon \left\|N\overline K\right\|\to \|NK\|$ is a principal $K$-bundle.
\end{lemma}
The following lemma is is a special case of a standard argument in the theory of simplicial sets (see \cite[Prop. 1.6.7]{Loday}), which applies in the case that there is an additional degeneracy.
\begin{lemma}
 $\left\|N\overline K\right\|$ is contractible.
\end{lemma}
\begin{proof}
By definition $\left\|N\overline K\right\|=\bigcup \Delta^n\times K^{n+1}/\sim$. Define $h:[0,1]\times\left\|N\overline K\right\|\to \left\|N\overline K\right\|$ by
 \[h_s(t_0,\dots,t_n,k_0,\dots,k_n)=(s,(1-s)t_0,\dots,(1-s)t_n,e,k_0,\dots,k_n).\]
$h_0((t_i),(k_i))=(0,(t_i),e,(k_i))=(\del^0((t_i)),e,(k_i))\sim ((t_i),\del_0(e,(k_i)))= ((t_i),(k_i))$ and 
$h_1((t_i),(k_i))=(1,(0),e,(k_i))=((\del^1)^n(1),e,(k_i))\sim (1,(\del_1)^n(e,(k_i)))= (1,e)\in \Delta^0\times K$. Hence the homotopy $h$ is a contraction of $\left\|N\overline K\right\|$ to a point.
\end{proof}
Thus we have a model of $EK=\left\|N\overline K\right\| \to BK=\faktor{EK}{K}=\|NK\|$. The map $\gamma\from N\overline K\to NK$ is a special case of the following object.

\begin{defn}[compare {\cite[p. 93]{Dupont}}]
 A simplicial $K$-bundle $\pi\colon E\to M$ is a sequence $\pi_p\colon E_p\to M_p$ of differential $K$-bundles, where $E=\{E_p\}$ and $M=\{M_p\}$ are simplicial manifolds, $\pi$ is a simplicial map and the right action of $K$ on $E$, $R_k:E\to E$, is simplicial, i.e., commutes with all face and degeneracy maps.
\end{defn}

A $G$-equivariant principal $K$-bundle $\pi\colon E\to M$ leads to the simplicial $K$-bundle \[\pi_\bullet=id_G\times \pi\from G^\bullet\times E\to G^\bullet\times M,\] where the action of $K$ is given by trivial extension along $G$. We want to construct a classifying map of the bundle $E$, which will be the geometric realization of a map of simplicial manifolds. Therefore we are going to define an intermediate bundle, mapping to the classifying space and to $G^\bullet\times E\to G^\bullet\times M$. This intermediate bundle is motivated by bisimplicial manifolds.

\subsection{Bisimplicial manifolds}
Our construction of the classifying map for simplicial bundles is motivated from the construction of the classifying map for non-simplicial bundles of \cite{Dupont}, which uses simplicial spaces. In the end it will turn out that we don't have to care about bisimplicial manifold, as they are reducible to their diagonal, which as simplicial manifold.
\begin{defn}[see {\cite[p. 196]{simphomtheory}}]
 A \emphind{bisimplicial set} is a simplicial object in the category of simplicial sets or equivalently a functor $\Delta^{op}\times\Delta^{op}\to Sets$.
\end{defn}
\begin{remark}A bisimplicial set is a collection of Sets $\{X_{p,q}|p,q=0,1,\dots\}$ together with vertical and horizontal face and degeneracy maps, such that vertical and horizontal maps commute.
 \end{remark}
\begin{defn}
 A \emph{bisimplicial manifold} is a collection of manifolds $\{X_{p,q}|p,q=0,1,\dots\}$, which forms a bisimplicial set and all face and degeneracy maps are smooth.
\end{defn}
We will construct a bisimplicial manifold from a simplicial manifold with a suitable cover.
\begin{defn}[see {\cite{Brylinski00,Gomi}}]\label{def:simpcover}
 A \emph{simplicial cover} \index{cover!simplicial} for the simplicial manifold $M_\bullet$ is a family $\mathcal U^\bullet=\{\mathcal U^{(p)}\}$\symindex[u]{\mathcal U^\bullet,\mathcal U^{(p)}} of open covers such that
\begin{enumerate}
 \item $\mathcal U^{(p)}=\{U^{(p)}_\alpha|\alpha\in A^{(p)}\}$ is an open cover of $M_p$, for each $p$, and
 \item the family of index sets forms a simplicial set $A^\bullet=\{A^{(p)}\}$, satisfying
 \item $\del_i(U^{(p)}_{\alpha})\subset U^{(p-1)}_{\del_i\alpha}$ and $\sigma_i(U^{(p)}_{\alpha})\subset U^{(p+1)}_{\sigma_i\alpha}$ for every $\alpha\in A^{(p)}.$
\end{enumerate}
\end{defn}
\begin{example}\label{cechcomplex}
 Let $M_\bullet$ be a simplicial manifold and $\mathcal U^\bullet$ a simplicial cover. 
Thus for any $p$ we can construct the Čech simplicial manifold \cite[Example 5, p.78]{Dupont}\symindex[N]{N_2M_{\mathcal U}}: 
\[(N_2M_{\mathcal U})_{p,q}\defeq\coprod_{(\alpha_0,\dots,\alpha_q)} U^{(p)}_{\alpha_0}\cap\dots\cap U^{(p)}_{\alpha_q}, \]
where the disjoint union is taken over all $(q+1)$-tuples $(\alpha_0,\dots,\alpha_q)\in (A^{(p)})^{q+1}$ with $U^{(p)}_{\alpha_0}\cap\dots\cap U^{(p)}_{\alpha_q}\neq \emptyset$. The face and degeneracy maps are given on the index sets by removing, respective doubling of the $i$-th index and on the open sets by the corresponding inclusions.

\begin{lemma}$N_2M_{\mathcal U}$ is a bisimplicial manifold.\end{lemma}
\begin{proof}Clearly, we have a bi-graded collection of manifolds. The third property of the simplicial cover ensure that the face and degeneracy maps of $M_\bullet$ restrict to the disjoint unions of intersections of covering sets and thus induce vertical face and degeneracy maps for the bisimplicial manifold. That these vertical maps compute with the horizontal `\Cech' maps follows, because it is the same, if one first restricts the neighborhood of a point, and then map the point, or doing it the other way around.\end{proof}
\end{example}

There are evidently two ways to geometrically realize a bisimplicial space: 1) first realize vertically and afterwards realize the received simplicial space in horizontal direction or 2) do it the other way around. Moreover, there is also a third one: realize the diagonal!
\begin{defn}[see {\cite[197]{simphomtheory}}]
Let $(M_{\bullet,\bullet},\del_i,\sigma_i,\del_i',\sigma_i')$ be a bisimplicial manifold. The \emphind{diagonal} is the simplicial manifold
\[(p\mapsto M_{p,p},\del_i\circ\del_i',\sigma_i\circ \sigma_i').\]
\end{defn}
\begin{lemma}[see {\cite[p.\,10/86/94]{Quillen}}]
 There are canonical homeomorphisms
 \[\|p\mapsto X_{p,p}\|=\|p\mapsto\|q\mapsto X_{p,q}\|\|=\|q\mapsto\|p\mapsto X_{p,q}\|\|.\]
\end{lemma}
This lemma motivates to reduce the bisimplicial manifold $N_2M_{\mathcal U}$ to its diagonal.
\begin{defn}\label{def:NMU}
Let $M_\bullet$ be a simplicial manifold with simplicial cover $\mathcal U^\bullet$. The simplicial manifold $NM_\mathcal{U}$ \symindex[N]{NM_\mathcal U} is defined to be the diagonal of $N_2M_{\mathcal U}$.\end{defn}

\begin{lemma}
    The inclusions $\mathcal U^{(p)}\ni U^{(p)}_\alpha\subset M_p$ induce a map 
\begin{align}\label{eq:cechmap}
NM_{\mathcal{U}} &\to M_\bullet,    
\end{align}
 which induces an functorial isomorphism in (complex) cohomology \[H^*(\|M\|)\cong H^*(\|NM_\mathcal{U}\|).\]
\end{lemma}
\begin{proof}
By \autoref{thm:realiso}, both sides are functorially isomorphic to the cohomology of the corresponding de Rham double complexes. We will prove the isomorphism in this model. Notice that differential forms do not only form a vector space, but a sheaf on each manifold, hence one can define a Čech resolution with respect to the cover $\mathcal U=\mathcal U^\bullet$:
\begin{align*}\check C^{p,q,r}(\mathcal U^\bullet)&=\check C^{p,q}(\mathcal U^\bullet,\Omega^r)\\
    &=\prod_{\alpha^{(p)}_0,\dots,\alpha^{(p)}_q\in A^{(p)}} \Omega^r\left(U^{(p)}_{\alpha^{(p)}_0}\cap\dots\cap U^{(p)}_{\alpha^{(p)}_q}\right)
\end{align*}
which is a triple complex with boundary map $d+(-1)^r\del+(-1)^{p+r}\delta$, where $\delta$ denotes the usually boundary in Čech cohomology (see, e.g., \cite[section III.4]{Hartshorne}). As differential forms form a fine sheave, i.e., admit a partition of unity, one can contract the triple complex in the Čech direction and obtains the de Rham double complex on $M_\bullet$

On the other hand, one has an equality of vector spaces \[\Omega^{p,r}(NM_{\mathcal{U}})=\check C^{p,p,r}({\mathcal{U^\bullet}}).\] In other words $\Omega^{\bullet,r}(NM_{\mathcal{U}})$ coincides on the set-level with the diagonal of the bisimplicial complex \[\check C^{\bullet,\bullet,r}({\mathcal{U}^\bullet}).\] Moreover, one checks from the definitions, that the face and degeneracy maps coincide, too. Thus  $\Omega^{\bullet,r}(NM_{\mathcal{U}})$ and the diagonal of $\check C^{\bullet,\bullet,r}({\mathcal{U}})$ are the same simplicial objects.

 By the generalized Eilenberg-Zilber theorem (\cite[Ch. IV, Theorem 2.4]{simphomtheory}), there is a chain homotopy equivalence between the diagonal and the total complex of a bisimplicial complex and hence from \[\bigoplus_{p+r=n} \Omega^{p,r}(NM_{\mathcal{U}}) \to \bigoplus_{p+q+r=n} \check C^{p,q,r}({\mathcal{U}}).\]
This induces the asserted isomorphism in cohomology.
\end{proof}

\subsection{The classifying map}
Let $\mathcal U=\{U_\alpha\}$ be an open cover of $M$ and let the induced open cover of $E$ be denoted by $\pi\invers\mathcal U=\{V_\alpha\}$, $V_\alpha=\pi^{-1}(U_\alpha)$. These covers of $M$ and $E$ induce simplicial covers of $G^\bullet\times E$ and $G^\bullet\times M$ as follows (compare \cite[p.319]{Gomi}):

 Define the simplicial index set $A^{(p)}=A^{p+1}$ with face and degeneracy maps given by removing respective doubling of the $i$-th element. Then define the simplicial cover $\pi^{-1}\mathcal U^{(p)}=\{V^{(p)}_\alpha\}_{\alpha\in A^{(p)}}$  inductively by
\[V^{(p)}_\alpha=\bigcap_{i=0}^{p}\del_i^{-1}\left(V^{(p-1)}_{\del_i(\alpha)}\right),\]
where $V^{(0)}_\alpha=V_\alpha$ for any $\alpha\in A^{(0)}=A$. The following lemma gives an alternative description of this construction.
\begin{lemma}
 \[V^{(p)}_\alpha=\left\{(g_1,\dots,g_p,m)\mid m\in V_{\alpha_p},g_pm\in V_{\alpha_{p-1}},\dots,g_1\dots g_pm\in V_{\alpha_0}\right\}\]
\end{lemma}
\begin{proof}We will prove this by induction. For $p=0$ there is nothing to show. Let $p>0$:
\begin{align*}V^{(p)}_\alpha&=\bigcap_{i=0}^{p}\del_i^{-1}(V^{(p-1)}_{\del_i(\alpha)})\\
    &=\bigcap_{i=0}^{p}\left\{(g_1,\dots,g_p,m)\mid \del_i(g_1,\dots,g_p,m)\in V^{(p-1)}_{\del_i(\alpha)}\right\}\\
\end{align*}
For any $i=0,\dots,p$ we can apply the induction hypothesis to $\del_i(g_1,\dots,g_p,m)\in V^{(p-1)}_{\del_i(\alpha)}\!$, what implies 
\begin{align*}m\in V_{\alpha_p}\;,\;\dots\;,\; g_{i+2}&\dots g_pm\in V_{\alpha_{i+1}},\\&g_i\dots g_pm\in V_{\alpha_{i-1}}\;,\;\dots\;,\; g_1\dots g_pm\in V_{\alpha_0}.\end{align*}
This is almost the right-hand side of the condition to be proven, just the $i$-th term is missing. As $i$ runs from $0$ to $p$, we get for the intersection exactly 
\[V^{(p)}_\alpha=\left\{(g_1,\dots,g_p,m)\mid m\in V_{\alpha_p},g_pm\in V_{\alpha_{p-1}},\dots,g_1\dots g_pm\in V_{\alpha_0}\right\}.\]
\end{proof}
By the construction of Definition \ref{def:NMU}, we obtain a simplicial bundle
\[\pi\colon N(G^\bullet\times E)_{\pi^{-1}\mathcal U}\to N(G^\bullet\times M)_{\mathcal U}.\]
and the commutative diagram
\[\begin{tikzcd}
 N(G^\bullet\times E)_{\pi\invers\mathcal U}\arrow{d}{\pi_{\mathcal U}}\arrow{r}&G^\bullet\times E\arrow{d}\\
N(G^\bullet\times M)_{\mathcal U}\arrow{r}&G^\bullet\times M
\end{tikzcd}\]
induced by the inclusions of the covering sets is a pullback, since the cover we take on $G^\bullet\times E$ is induced by $\pi$ and $\mathcal U^\bullet$.

Suppose the cover $\mathcal U=\{U_\alpha\}_{\alpha\in A}$ of $M$ trivializes $E$ with trivialization \[\varphi_\alpha\from V_\alpha=\pi^{-1}(U_\alpha)\to U_\alpha\times K\] and transition functions $g_{\alpha\beta}\colon U_\alpha\cap U_\beta\to K$. Then there is an induced map \[\overline{\psi}\from N(G^\bullet\times E)_{\pi^{-1}\mathcal U}\to N\overline K,\] which is given on the intersection of $p+1$ covering sets of $G^p\times E$ 
\[V=\bigcap_{j=0}^p V^{(p)}_{\alpha^j_0,\dots,\alpha^j_p}\]
by 
\[(g_1,\dots,g_p,x)\mapsto (\varphi_{\alpha^0_0}(g_1\dots g_px),\varphi_{\alpha^1_1}(g_2\dots g_px),\dots,\varphi_{\alpha^p_p}(x))\in K^{p+1},\]
where, on the right-hand side, the maps $\varphi_\alpha$ are understood to be composed with the projection to $K$.

Next, we want to define $\psi:N(G^\bullet\times M)_{\mathcal U}\to NK$, such that $\overline\psi$ covers $\psi$. Therefore we need some additional transition functions of the bundle. Define 
\begin{align*}
h_{\alpha\beta}\colon G\times M\supset\del_0\invers U_\alpha\cap\del_1\invers U_\beta&\to K\\
(g,m)&\mapsto (\pi_2\circ \varphi_\alpha(gx))(\pi_2\circ \varphi_\beta(x))\invers,
\end{align*}
for any $x\in \pi\invers(m)$. This definition is independent of the chosen fiber element, as any other element in the fiber equals $xk$ for some $k\in K$ and 
\begin{align*}
 \left(\pi_2\circ \varphi_\alpha(g(xk))\right)(\pi_2\circ \varphi_\beta(xk))\invers&=(\pi_2\circ \varphi_\alpha(gx)k)(\pi_2\circ \varphi_\beta(x)k)\invers\\&=(\pi_2\circ \varphi_\alpha(gx))kk\invers(\pi_2\circ \varphi_\beta(x))\invers.
\end{align*}
As there exist local smooth sections, $G$ acts smoothly and the trivialization maps are smooth, $h_{\alpha\beta}$ is smooth, too.

Define $\psi$ on \[U=\bigcap_{j=0}^q U^{(p)}_{\alpha^j_0,\dots,\alpha^j_p}\] by 
\begin{multline}(g_1,\dots,g_p,m)\mapsto (h_{\alpha^0_0\alpha^1_1} (g_1,g_2\dots g_pm),\\h_{\alpha^1_1\alpha^2_2} (g_2,g_3\dots g_pm),\dots,h_{\alpha^{p-1}_{p-1}\alpha^p_p} (g_p,m),*).\end{multline}
The definitions of $\psi,\overline\psi$ and $\gamma$ yield a commutative diagram
\begin{equation}\label{eq:classmap}
\begin{tikzcd}
 N(G^\bullet\times E)_{\pi\invers\mathcal U}\arrow{d}{\pi_{\mathcal U}}\arrow{r}{\bar\psi}&N\overline K\arrow{d}{\gamma}\\
N(G^\bullet\times M)_{\mathcal U}\arrow{r}{\psi}&NK
\end{tikzcd}
\end{equation}
of simplicial manifolds. Later on we will need the following statement.
\begin{lemma}\label{lem:classmapsquare}
 The geometric realization of this diagram is a pullback.
\end{lemma}
\begin{proof}
    This follows, as the bundle map is $K$-equivariant, compare, e.g., Proposition 8.6 of {\cite[Ch. I]{tomDieck}}.
\end{proof}
\section{The main results}
\subsection{Dupont's simplicial forms, connections and transgression}\label{sec:Dupontforms}
Let $M_\bullet$ be a simplicial manifold. Dupont has given another definition for simplicial differential forms than the de Rham complex $\Omega^{\bullet,*}$.
\begin{defn}[Def. 6.2 of \cite{Dupont}]
 A \emphind{simplicial $n$-form} $\omega$ on $M_*$ is a sequence \symindex[o]{\omega^{(p)}}$\omega=\{\omega^{(p)}\}_{p\in \Natural}$, where each $\omega^{(p)}\in\Omega^n(\Delta^p\times M_p)$, such that
\[(\del^i\times \id_{M_p})^*\omega^{(p)}=(\id_{\Delta^{p-1}}\times\del_i)^*\omega^{(p-1)}\]
on $\Delta^{p-1}\times M_p$ for $i=0,\dots,p$ and $p\in \Natural$. The space of simplicial $n$-forms will be denoted by $\mathcal A^n(M_\bullet)$\symindex[A]{\mathcal A^*}.
\end{defn}
The differential $d\from\mathcal A^n(M_\bullet)\to\mathcal A^{n+1}(M_\bullet)$ is in any simplicial level the sum of $d_{M_p}+d_{\Delta^p}$. We can compare this complex of differential forms to the simplicial de Rham complex.
\begin{lemma}[Theorem 6.4 of {\cite{Dupont}}] The map
\begin{align*}
 \int_\Delta:\mathcal A^n(M_p)&\to \bigoplus_{p+q=n} \Omega^{q}(M_p)\\
\omega^{(p)}&\mapsto \int_{\Delta^p} \omega^{(p)}\in \Omega^{n-p}(M_p),
\end{align*}
given by integration over the simplices, induces a quasi-isomorphism of chain complexes.
\end{lemma}
Recall that a connection on a principal $K$-bundle $E$ is 1-form $\vartheta\in \Omega^1(E,\mathfrak k)$ with values in the Lie algebra, which is $K$-equivariant and a section of the differential of the $K$-action $E\times \mathfrak k\to TE$.
\begin{defn}
A \emph{connection} \index{connection!simplicial principal $K$-bundle} on a simplicial principal $K$-bundle $E\to M$ is a Dupont-1-form $\vartheta\in \mathcal A^1(E,\mathfrak k)$ \symindex[t]{\vartheta} such that the restriction to any $\Delta^p\times E_p$ is a connection on the bundle 
\[\Delta^p\times E_p\to\Delta^p\times M_p.\]
The curvature of the connection is defined as
\[\Omega=d\vartheta+\frac12 [\vartheta,\vartheta]\in\mathcal A^2(E,\mathfrak k),\]
where $[\vartheta,\vartheta]$ denotes the image of $\vartheta\wedge \vartheta$ under the Lie bracket $[\cdot,\!\cdot]\from \mathfrak g\otimes \mathfrak g \to \mathfrak g$.
\end{defn}
The Chern-Weil construction turns over to the simplicial setting (compare \cite{Dupont}):
\begin{theorem}Let $P\in I^q(K)$ be an invariant symmetric polynomial, $E_\bullet\to M_\bullet$ a simplicial principal $K$-bundle with connection $\vartheta$ and curvature $\Omega$.\label{thm:chernweil}
    \begin{enumerate}
        \item $P(\Omega^q)=P(\Omega,\dots,\Omega)\in \mathcal A^{2q}(E)$ is a basic $2q$-form, i.e., it is an element of $\pi^*\mathcal A^{2q}(M)$ or in words a pullback from the base space $M$. The form $\omega_P(\vartheta)\in \mathcal A^{2q}(M)$, s.t., $\pi^*\omega_P(\vartheta)=P(\Omega^q)$, is called \emph{characteristic form} of $(E,\vartheta)$. \index{characteristic form!Dupont model}
		\item $I^q(K)\ni P\mapsto \omega_P(\vartheta)\in \mathcal A^{2q}(M)$ is an algebra homomorphism.
		\item $\omega_P(\vartheta)\in \mathcal A^{2q}(M)$ is closed.	
		\item Given two simplicial connections $\vartheta_0,\vartheta_1$ on $E_\bullet$, then there is a path of connections from the first to the second, i.e., a connection $\tilde{\vartheta}$ on $\Real\times E_\bullet\to \Real\times M_\bullet$ such that $\tilde{\vartheta}|_{\{i\}\times M}=\vartheta_i$ for $i=0,1$. The \emph{transgression form}\index{transgression form!simplicial}\symindex[o]{\widetilde{\omega}} \[\widetilde{\omega}_P(\vartheta_1,\vartheta_0)=\int_{[0,1]\times M/M} \omega_P(\tilde{\vartheta}) \in \mathcal A^{2q-1}(M)/d\mathcal A^{2q-2}(M),\] is independent of the path chosen and satisfies as well \begin{align}d\tilde{\omega}_P(\vartheta_1,\vartheta_0)&=\omega_P(\vartheta_1)-\omega_P(\vartheta_0)\label{eq:trans1}\end{align} as for any third connection $\vartheta_2$ on $E$\begin{align}\tilde{\omega}_P(\vartheta_2,\vartheta_1)+\tilde{\omega}_P(\vartheta_1,\vartheta_0)&=\tilde{\omega}_P(\vartheta_2,\vartheta_1).\label{eq:trans2}\end{align}
		\item Let $f\from N\to M$ be a smooth map, then $f^*\omega_P(\vartheta)=\omega_P((\bar f,f)^*\vartheta)$.
		\item The class $c_P(E)=[\omega_P(\vartheta)]\in H^{2n}(M)$ \symindex[c]{c_P} is called \emphind{characteristic class} (defined via the Chern-Weil-Construction). It is independent of the connection and does only depend on the isomorphism class of the bundle.
    \end{enumerate}
\end{theorem}
\begin{proof}
 These statements are more or less standard, but we give proofs for convenience of the reader.

First, the simplicial form is basic, if this is true on any simplicial level, where it is a standard fact that the curvature is horizontal and equivariant (see e.g. \cite[Prop 3.12 b)]{Dupont}) and as $P$ is invariant, $P((\Omega^{(p)})^q)$ is basic. To the second assertion: the question about sums and scalars follows clearly from the definition. Let $Q\in I^l(K)$, then \[(P\cdot Q)(\Omega^{q+l})= \frac{1}{(q+l)!}\sum_{\sigma} P(\Omega^{q})\wedge Q(\Omega^l) =P(\Omega^{q})\wedge Q(\Omega^l).\]Thus $P\mapsto \omega_P(\vartheta)$ is a homomorphism of algebras.

To show closedness, it is, since $\pi^*\from\mathcal A^{2q}(M)\to\mathcal A^{2q}(E)$ is injective, sufficient to show that $dP(\Omega^q)\in\mathcal A^{2q+1}(E)$ vanishes. We compute
	\begin{equation}\label{eq:calcdP}\begin{aligned}
		dP(\Omega^q)&=qP(d\Omega\wedge\Omega^{q-1})\quad\text{(as P is symmetric)}\\
	    &=qP((\frac12 d[\vartheta,\vartheta])\wedge\Omega^{q-1})\\
		&=qP(([d\vartheta,\vartheta])\wedge\Omega^{q-1})\quad\text{(as } d[\vartheta,\vartheta]=[d\vartheta,\vartheta]-[\vartheta,d\vartheta])\\
		&=qP(([\Omega,\vartheta])\wedge\Omega^{q-1})\quad\text{(as } [[\vartheta,\vartheta],\vartheta]=0 \text{ by the Jacobi identity)}.
	\end{aligned}
	\end{equation}
	On the other hand, $P$ is $K$-invariant. Let $X,Y_1,\dots,Y_q\in \mathfrak k$ and differentiating the equation 
	\[P(Y_1,\dots,Y_q)=P(\Ad(\exp(tX))Y_1,\dots,\Ad(\exp(tX))Y_q)\]
	by $t$ at $t=0$ yields
		\[0=\sum_{i=1}^q P(Y_1,\dots,[X,Y_i],\dots,Y_q)=\sum_{i=1}^q P([X,Y_i],Y_1,\dots,\hat Y_i,\dots,Y_q),\]
	what shows $P(([\Omega,\vartheta])\wedge\Omega^{q-1})=0$. Thus $dP(\Omega^q)=0$ by equation \eqref{eq:calcdP}.

To prove statement four about the transgression, define the connection $\tilde{\vartheta}_t=(1-t)\pr_E^*\vartheta_0+t\pr_E^*\vartheta_1$ on $\Real\times E$, which is obviously a path of connections from $\vartheta_0$ to $\vartheta_1$. Thus \eqref{eq:trans1} follows from Stokes theorem (applied to any simplicial level). Let $A^2=\{x_0+x_1+x_2=1\}\subset\Real^3$ be the hyperplane, whose intersection with the positive octant is $\Delta^2$. Define by $\hat{\vartheta}=\sum_ix_i\vartheta_i$ a the connection on $A^2\times E_\bullet\to A^2\times M_\bullet$. By Stokes one has \begin{align}\tilde{\omega}_P(\vartheta_2,\vartheta_1)+\tilde{\omega}_P(\vartheta_1,\vartheta_0)-\tilde{\omega}_P(\vartheta_2,\vartheta_1)=d\int_{\Delta^2}\omega_P(\hat{\vartheta}),\label{eq:trans3}\end{align} from which \eqref{eq:trans2} follows. To show the independence from the chosen path, take $\vartheta_2=\vartheta_1$ and define another connection $\hat{\vartheta}$ on $A^2\times E_\bullet\to A^2\times M_\bullet$ in the following way: $\hat{\vartheta}$ restricts to $\vartheta_i$ on the $i$-th vertex of the simplex (which is the intersection of the hyperplane with the non-negative octant), it is constantly $\pr_E^*\vartheta_1$ on the edge $(1,2)$ (from vertex $1$ to vertex $2$), the convex combination on $(0,1)$, an arbitrary path on the last edge and an interpolation in the interior (say the convex combination on lines parallel to $(1,2)$). Then \eqref{eq:trans3} implies the independence of the path, because $\tilde{\omega}_P(\vartheta_2,\vartheta_1)$ is zero as an integral over a pullback form.

The 5th statement, about pullbacks, follows from the two facts that, firstly, the curvature of the pullback connection $(\bar f,f)^*\vartheta$ is $(\bar f,f)^*\Omega$, i.e. the pullback of the curvature and, secondly, pullbacks are an algebra homomorphism on differential forms. For the last assertion: the difference of the characteristic forms for two connections is an exact form by 4., thus the class is independence of the connection. That the class only depends on the isomorphism class follows from 5. applied to an isomorphism of the bundles, which covers the identity map on the base space.
\end{proof}

\subsection{Equivariant characteristic forms}
Let $G$ and $K$ be Lie groups and $E$ a smooth $G$-equivariant principal $K$-bundle over a smooth manifold $M$ with $G$-invariant connection $\vartheta\in\Omega^1(E,\mathfrak k)^G$. 
\begin{defn}[see {\cite[p.543]{BerlineVergne83}}]\label{Def:momentmap}
	The \emphind{moment map} of $\vartheta$ is defined as \symindex[m]{\mu}
\[\mu^\vartheta\from \mathfrak g\ni X\mapsto \iota(X^\sharp)\vartheta\in C^\infty(E,\mathfrak k)^K.\]
\end{defn}
\begin{lemma}
The moment map is $G$- and $K$-equivariant.
\end{lemma}
\begin{proof}
    The $K$-equivariance of $\mu$ follows from the $K$-equivariance of $\vartheta$. Now let $g\in G$, then
\begin{align*}\mu(\Ad_gX)(gx)&=\vartheta(gx)\left[\ddt \left(\exp(gXg\invers)(gx)\right)\right]\\
    &=\vartheta(gx)\left[g\ddt \left(\exp(X)x\right)\right]\\
&=\vartheta(gx)[gX^\sharp]\\
&=\vartheta(x)[X^\sharp]\quad\text{by the $G$-invariance of $\vartheta$}\\
&=\mu(X)(x).
\end{align*}
\end{proof}

Given $\vartheta\in\Omega^1(E,\mathfrak k)^G$, one defines a simplicial connection on $G^\bullet\times E$ as follows (compare \cite[p.104]{Getzler}):
Let $\vartheta_i$ be the pullback of $\vartheta$ to $\Delta^p\times G^p\times E$ along
\begin{align*}
 \Delta^k\times G^k\times E&\to E\\
(t_0,\dots,t_k,g_1,\dots,g_k,e)&\mapsto g_{i+1}\dots g_ke.
\end{align*}
Now define $\Theta\in\mathcal A^1(G^\bullet\times E,\mathfrak k)$ on $G^\bullet\times E$ by
\begin{align}\label{eq:Theta}\Theta^{(p)}&=t_0 \vartheta_0 + t_1\vartheta_1+ \dots+t_p\vartheta_p,\end{align}
where the $t_i$ are the barycentric coordinates on the simplex.
These forms satisfy
\[(\del^i\times \id_{G^p\times E})^*\Theta^{(p)}=(\id_{\Delta^{p-1}}\times\del_i)^*\Theta^{(p-1)}\]
and hence $\Theta$ is a simplicial Dupont one form. 

We are now going to calculate the characteristic form of the simplicial connection $\Theta$. We will see that this actually leads to the equivariant characteristic form of $\vartheta$ as defined by Berline and Vergne, i.e., one replaces the curvature by the sum of the curvature and the moment map. This is a more detailed reformulation of \cite[Section 3.3.]{Getzler}.
\begin{theorem}\label{thm:classform} Let $P\in I^*(K)$ be an invariant symmetric polynomial.
    \[\pr_0\left(\mathcal J\left(\int_\Delta \omega_P(\Theta)\right)\right)= P(\Omega^\vartheta+\mu^\vartheta)\in S^*(\mathfrak g\dual)\otimes\Omega^*(M)\]
Here $\mathcal J$ is the map defined in \ref{def:Getzlermap} and $\pr_0$ is the projection from the Getzler complex to its zeroth simplicial level. As $\vartheta$ is $G$-invariant, the equation actually holds in $\Omega_G(M)=\left(S^*(\mathfrak g\dual)\otimes\Omega^*(M)\right)^G$.
\end{theorem}
\begin{proof}
Let  \[\Omega=d\Theta +\frac12[\Theta,\Theta]\]
denote the curvature of $\Theta$. We should refine the grading of the simplicial Dupont forms (compare \cite[p.91]{Dupont}): A form in $\mathcal A^*(G^\bullet\times M)$ restricts to a differential form on the product $\Delta^p\times (G^p\times M)$ for each $p$. Thus we can grade the form by the differential form degree  on the simplex $\Delta^p$ part and on the form degree on the manifold part $G^p\times M$, thus the degree defined before is the sum of both. By construction, the form degree of $\Theta$ in direction of the simplex is zero. Thus the form degree of $\Omega$ in simplex direction can be at most one. Therefore
\[\int_{\Delta}\Omega=\Omega_0+\Omega_1\in\Omega^2(E,\mathfrak k)\oplus \Omega^1(G\times E,\mathfrak k).\]
Here $\Omega_0=d\Theta^{(0)} +\frac12[\Theta^{(0)},\Theta^{(0)}]=d\vartheta+\frac12[\vartheta,\vartheta]$ is exactly the curvature of $\vartheta$. While
\begin{align*}
 \Omega_1&=\int_{\Delta^1} d\Theta^{(1)} +\frac12[\Theta^{(1)},\Theta^{(1)}]\\
&=\int_{\Delta^1} dt_0 \vartheta_0+dt_1 \vartheta_1\\
&=\int_0^1 dt (\vartheta_0- \vartheta_1)\\
&=\vartheta_0- \vartheta_1\\
&=\del_1^*\vartheta-\del_0^*\vartheta.
\end{align*}
To investigate this further, let $X=X_G+X_E$ be a vector field on $G\times E$, decomposed in the directions of $G$ and $E$, then
\begin{align*}
(\Omega_1)(g,p)[X]&=(\del_1^*\vartheta)(g,p)[X]-(\del_0^*\vartheta)(g,p)[X]\\
&=\vartheta(gp)[(T\del_1)(X)]-\vartheta(p)[(T\del_0)X]\\
&=\vartheta(gp)[X_G^\sharp]+\vartheta(gp)[gX_E]-\vartheta(p)[X_E]\\
&=\vartheta(p)[g\invers (X_G^\sharp(gp))],
\end{align*}
where $X_G^\sharp$ denotes the fundamental vector field of $X_G$. Restricting this to $e\in G$, what is the same as applying the map $\pr_0\mathcal J$, one obtains the definition of the moment map $\mu^\vartheta=\iota(X^\sharp)\vartheta$. 

The statement of the theorem now follows from the next lemma.
\end{proof}
\begin{lemma}The composition of maps
    \[\pr_0\circ \mathcal J\circ\int_\Delta\from \mathcal A^*(G^\bullet\times M)\to S^*(\mathfrak g\dual)\otimes\Omega^*(M)\] is a homomorphism of algebras.
\end{lemma}
\begin{proof}
The map is clearly a homomorphism of vector spaces. Hence we only have to show that
\[\pr_0\mathcal J\int_\Delta \omega_1\wedge \omega_2=\left(\pr_0\mathcal J\int_\Delta \omega_1\right)\wedge \left(\pr_0\mathcal J\int_\Delta \omega_2\right)\]
for $\omega_i\in \mathcal A^*(G^\bullet\times M)$. Using the refined grading defined above, we can, by additivity of the map, restrict ourselves to $\omega_i$ being a $p_i$-form in the direction of the simplex and a $q_i$-form in the direction of $G^p\times M$. Without loss of generality $\omega_i^{(p_i)}=dt_1\wedge\dots\wedge dt_{p_i}\wedge \bar{\omega_i}$

Let $X\in \mathfrak g$. We calculate:
\begin{align*} 
&\left(\pr_0\mathcal J\int_\Delta \omega_1\wedge \omega_2\right)(X)\\
&=i_M^*\iota_1(X)\dots\iota_{p_1+p_2}(X)\int_\Delta \omega_1\wedge \omega_2 \\
&=i_M^*\int_{\Delta^{p_1+p_2}} \iota_1(X)\dots\iota_{p_1+p_2}(X) (\omega_1\wedge \omega_2)\\
&=i_M^*\int_{\Delta^{p_1+p_2}} \sum_{\pi\in S(p_1,p_2)} (-1)^{f(\pi)}\left(\iota_{\pi(1)}(X)\dots\iota_{\pi(p_1)}(X)\omega_1^{(p_1+p_2)}\right)\wedge\\&\omit\hfill $\left(\iota_{\pi(p_1+1)}(X)\dots\iota_{\pi(p_1+p_2)}(X)\omega_2^{(p_1+p_2)}\right)$\\
\intertext{Here $S(p_1,p_2)$ is the shuffle group (see Definition \ref{def:Getzlermap}). What is the sign $f(\pi)$? Take a shuffle $\pi$, let $c_\pi(k)$ be the number of indices in $\{\pi(1),\dots,\pi(p_1)\}$, which are larger than $\pi(p_1+p_2-k)$. Then we get, when expanding the contraction, a $(-1)^{p_1+q_1-c_\pi(k)}$ for swapping $\iota_{\pi(p_1+p_2-k}(X)$ with the partially contracted $\omega_1$. Summing this up for $k=0,\dots,p_2-1$ results in $f(\pi)$. On the other hand the sum of the $c_\pi(k)$ is exactly the number of inversions of $\pi$ as the two groups of the shuffle are in order, thus $(-1)^{f(\pi)}(-1)^{(q_1-p_1)p_2}\sgn(\pi)=1$.}
&=i_M^*\!\int_{\Delta^{p_1+p_2}}\! \sum_{\pi} (-1)^{f(\pi)}(\iota_{\pi(1)}(X)\dots\iota_{\pi(p_1)}(X)(\id_\Delta\times(\del_{\pi(p_1+1)}\dots\del_{\pi(p_1+p_2)})\!)^*\omega_1)\wedge\\&\omit\hfill $ (\iota_{1}(X)\dots\iota_{p_2}(X)(\id_\Delta\times(\del_{\pi(1)}\dots\del_{\pi(p_1)}))^*\omega_2)$\\
\intertext{We can add the $\del$'s in front of the $\omega_i$ as $\sigma_i^*\del_i^*=\id$, and $\sigma_i\circ i_M=i_M$. Now apply the property of simplicial Dupont forms with respect to the face maps.}
&=i_M^*\int_{\Delta^{p_1+p_2}} \sum_{\pi} (-1)^{f(\pi)}(\iota_1(X)\dots\iota_{p_1}(X)((\del^{\pi(p_1+1)}\dots\del^{\pi(p_1+p_2)})\times\id)^*\omega_1^{(p_1)})\wedge\\&\omit\hfill $ (\iota_1(X)\dots\iota_{p_2}(X)((\del^{\pi(1)}\dots\del^{\pi(p_1)})\times \id)^*\omega_2^{(p_2)})$\\
&=i_M^*\int_{\Delta^{p_1+p_2}} \sum_{\pi} (-1)^{f(\pi)}(\iota_1(X)\dots\iota_{p_1}(X)dt_{\pi(1)}\wedge\dots\wedge dt_{\pi(p_1)}\wedge \bar{\omega_1})\wedge\\&\omit\hfill $ (\iota_1(X)\dots\iota_{p_2}(X)dt_{\pi(p_1+1)}\wedge\dots\wedge dt_{\pi(p_1+p_2)}\wedge \bar{\omega_2})$\\
&=i_M^*\int_{\Delta^{p_1+p_2}} \sum_{\pi} (-1)^{f(\pi)}(dt_{\pi(1)}\wedge\dots\wedge dt_{\pi(p_1)}\wedge dt_{\pi(p_1+1)}\wedge\dots\wedge dt_{\pi(p_1+p_2)}\wedge\\&\omit\hfill $ (-1)^{(q_1-p_1)p_2+p_1^2+p_2^2} (\iota_1(X)\dots\iota_{p_1}(X)\bar{\omega_1})\wedge (\iota_1(X)\dots\iota_{p_2}(X)\bar{\omega_2})$\\
\intertext{The sign comes from rearranging the forms. Now recall that the volume of the $p$-simplex is $\frac1{p!}$ and the number of elements of the shuffle group is $\frac{(p_1+p_2)!}{p_1! p_2!}$. Thus we obtain:}
&=\frac{1}{p_1! p_2!}\, (-1)^{p_1^2+p_2^2}\, i_M^*(\iota_1(X)\dots\iota_{p_1}(X)\bar{\omega_1})\wedge(\iota_1(X)\dots\iota_{p_2}(X)\bar{\omega_2})\\
&=\left(i_M^*\int_\Delta \iota_1(X)\dots\iota_{p_1}(X) \omega_1\right)\wedge\left(i_M^*\int_\Delta \iota_1(X)\dots\iota_{p_2}(X) \omega_2\right)\\
&=\left(\pr_0\mathcal J\int_\Delta \omega_1\right)(X)\wedge \left(\pr_0\mathcal J\int_\Delta \omega_2\right)(X)\\
&=\left(\left(\pr_0\mathcal J\int_\Delta \omega_1\right)\wedge \left(\pr_0\mathcal J\int_\Delta \omega_2\right)\right)(X).
\end{align*}
This finishes the proof.
\end{proof}

This theorem motivates the following definition, which translates the definition of characteristic form and transgression to the Cartan model.
\begin{defn}[compare {\cite[543]{BerlineVergne83}}]
    Let $\vartheta,\vartheta'$ be $G$-invariant connections on a $G$-equivariant principal $K$-bundle $E\to M$. The \emph{characteristic form}\index{characteristic form!Cartan model} related to the invariant polynomial $P\in I^*(K)$ is defined as $P(\Omega^\vartheta+\mu^\vartheta)\in \Omega_G^*(M)$ and the \emph{transgression form}\index{transgression form! Cartan model} is $\widetilde{\omega}_P(\vartheta',\vartheta)=\int_I P(\Omega^{\vartheta_t}+\mu^{\vartheta_t})$,\symindex[o]{\widetilde{\omega}_P} where $\vartheta_t$ is the convex combination of $\vartheta$ and $\vartheta'$ as in the proof of Theorem \ref{thm:chernweil}.
\end{defn}
Let $\Theta$ and $\Theta'$ be the connections associated to $\vartheta$ and $\vartheta'$, respectively, by \eqref{eq:Theta}. The following lemma shows that we did not run into notational difficulties.
\begin{lemma}
    We have \[\widetilde{\omega}_P(\vartheta',\vartheta)=\pr_0\left(\mathcal J\left(\int_\Delta\widetilde{\omega}_P(\Theta',\Theta)\right)\right),\] where the right-hand side is the transgression form of Theorem \ref{thm:chernweil}.
\end{lemma}
\begin{proof}
As the composition of maps $\pr_0\circ \mathcal J\circ\int_\Delta$ is linear and all integral are taken over compact sets, it commutes with $\int_I$ and thus the simplicial transgression form is mapped to the one in the Cartan model.
\end{proof}
\subsection{Universal Chern-Weil construction}
There is a `canonical' connection on the classifying bundle (compare \cite[p.94]{Dupont}). As seen above $EK\to BK$ is the geometric realization of the simplicial bundle $\gamma\colon N\overline K\to NK$. Let $\vartheta_0\in \Omega^1(K,\mathfrak k)$ denote the unique connection of the trivial bundle $K\to \mathrm{pt}$, i.e., \[\vartheta_0(k)=L_{k\invers}\from T_kK\to T_eK=\mathfrak k.\] 
Let \[\pi_{i}\from \Delta^p\times K^{p+1}\to K\] denote the projection to the $i$-th coefficient, $i=0,\dots,p$ and $\vartheta_{i}=\pi_{i}^*\vartheta_0$. Then we define $\bar{\vartheta}$ on $\Delta^p\times (N\overline K)_{p}$ by \[\bar{\vartheta}=\sum_{i} t_i\vartheta_{i},\] where $(t_0,\dots,t_p)$ are barycentric coordinates on the simplex. $\bar{\vartheta}|_{\Delta^p\times (N\overline K)_p}$ is a connection on $\Delta^p\times (N\overline K)_p$, as it is a convex combination of connections. It can be seen easily from the definition, that $\bar{\vartheta}$ is a  simplicial Dupont 1-form. \label{NGconnection}

Now we are able to relate the universal classes in $H^*(BK)$ with those constructed by the Chern-Weil construction. This is a slight generalization of \cite[Theorem 6.13]{Dupont}.
\begin{theorem}\label{thm:unichernweil}Let $K$ be a Lie group.
 \begin{enumerate}
     \item There is a homomorphism \begin{align*}c\from I^*(K)&\to H^{2*}(BK)=\faktor{\mathcal A^{2*}(NK)_\cl}{d\mathcal A^{2*-1}(NK)}\\ P&\mapsto c_P(N\overline K)=[\omega_P(\bar{\vartheta})],\end{align*} which is, for compact groups, inverse to $\pr_0\circ \mathcal J\circ\int_\Delta$.
	\item Let $G$ be another Lie group, $P\in I^*(K)$ and let $E$ be a smooth $G$-invariant principal-$K$-bundle over the smooth manifold $M$, then \[c_P(E)=c(P)(EG\times_G E),\]
i.e., the definition via the Chern-Weil construction on the left-hand side equals the pullback definition via the universal class $c(P)=c_p(N\overline K)$ on the right hand side.
 \end{enumerate}
\end{theorem}
\begin{proof}
    The homomorphism property is just a special case of Theorem \ref{thm:chernweil} and the statement about the inverse follows from Theorem \ref{thm:classform}, as the curvature of $\vartheta$ is equal to zero and the moment map equals the identity.

 For the second assertion let $\mathcal U$ be a trivializing cover of $\pi\colon E\to M$. From \eqref{eq:classmap} and \eqref{eq:cechmap} we obtain a commutative diagram
\begin{equation}
\begin{tikzcd}
 G^\bullet\times E\arrow{d}&\arrow{l}{i}N(G^\bullet\times E)_{\pi\invers\mathcal U}\arrow{d}{\pi_{\mathcal U}}\arrow{r}{\bar\phi}&N\overline K\arrow{d}{\gamma}\\
G^\bullet\times M&\arrow{l}{i}N(G^\bullet\times M)_{\mathcal U}\arrow{r}{\phi}&NK.
\end{tikzcd}
\end{equation}
Now:
\begin{alignat*}{2}
 i^*c(P)(EG\times_G E)&=c(P)(\|N(G^\bullet&&\times E)_{\pi\invers\mathcal U}\|)\\
&=\|\varphi\|^*c(P) && \text{(right square is a pullback by Lemma \ref{lem:classmapsquare})}\\
&=\varphi^*[\omega_P(\bar{\vartheta})]&& \text{(isomorphism of Theorem \ref{thm:realiso})}\\
&=[\omega_P(\varphi^*\bar{\vartheta})]&&\\
&=[\omega_P(i^*\vartheta)]&& \text{(class is independent of connection)}\\
&=i^*c_P(E) && \text{(Chern-Weil definition)}   
\end{alignat*}
The assertion follows, since $i^*$ is an isomorphism.
\end{proof}
\begin{defn}[compare Def. 2.39 of {\cite{Bunke}}]\label{def:charform}
    A \emphind{characteristic form} $\omega$ for principal $K$-bundles of degree $n$ associates to each connection $\vartheta\in\Omega^1(E,\mathfrak k)$ on a smooth principal $K$-bundles $\pi\from E\to M$ a closed differential form $\omega(\vartheta)\in\Omega_\cl^n(M)$, such that $\omega((\bar f,f)^*\vartheta)=f^*\omega(\vartheta)$ for every pullback-diagram 
\[\begin{tikzcd}
E'\arrow{r}{\bar f}\arrow{d}{\pi'}&E\arrow{d}{\pi}\\
M'\arrow{r}{f}&M.
\end{tikzcd}\]
\end{defn}

\begin{cor}\label{cor:CharformPoly}
Let $K$ be a compact Lie group. Any characteristic form $\omega$ for principal $K$-bundles of degree $2n$ corresponds to a polynomial $P\in I^n(K)$, which induces via the Chern-Weil construction the same characteristic class as $\omega$.    
\end{cor}
\begin{proof}
	Let $\bar{\vartheta}=\left\{\bar{\vartheta}^{(p)}\right\}_{p\in\Natural}$ denote the simplicial connection on $N\overline K\to NK$ defined above. By the pullback property of the characteristic form, $\left\{\omega\left(\bar{\vartheta}^{(p)}\right)\right\}_{p\in\Natural}\in\mathcal A^n_\cl(NK)$ is a closed Dupont $n$-form. Thus $\mathcal J\circ\int_\Delta (\omega)$ is a cocycle in the Getzler model. Since $K$ is compact, the Getzler model contracts to the Cartan model and thus $\mathcal J\circ\int_\Delta (\omega)$ is cohomologous to an element of $I^n(K)$, which we will call $P$. Since $N\overline K\to NK$ is a universal principal $K$-bundle, the assertion follows.
\end{proof}

\begin{remark}
    There is a slightly shorter way, to show the main theorem above, then the one we gave, but we will apply the construction given above in the companion paper \cite{Ich2}, when defining differential refinements of equivariant characteristic classes. As this shorter construction is maybe interesting to the reader we will give a sketch (which generalizes arguments of \cite{DupontKamber}). Instead of the simplicial manifold of the covering one defines a special classifying space for each bundle: As before let $\pi:E\to M$ be an $G$-equivariant principal $K$-bundle and $G^\bullet\times E\to G^\bullet\times M$ the associated simplicial bundle. We define a simplicial manifold $N(G^\bullet\times E)^\bullet$, s.t., $N(G^p\times E)^p=G^p\times E^{p+1}$, face and degeneracy maps act on the $G$ part as described above and on the $E$ part be removing/doubling the $i$th entry. In particular $\del_p(g_1,\dots,g_p,b_0,\dots,b_p)=(g_1,\dots,g_{p-1},g_pb_0,\dots,g_pb_{p-1})$. $K$ acts diagonal on the $E$'s. This allows to define maps
	\[\begin{tikzcd}
    G^\bullet\times E\arrow{d}{\pi}\arrow{r}{\bar{\psi}}&N(G^\bullet\times E)^\bullet\arrow{d}{}&\arrow[swap]{l}{\bar{\varphi}_x}N\overline K\arrow{d}{\gamma}\\
    G^\bullet\times M\arrow{r}{\psi}&N(G^\bullet\times E)^\bullet/K&\arrow[swap]{l}{\varphi_x}NK.
	\end{tikzcd}\]
	Here $\bar{\psi}$ is induced by the diagonal inclusion $E\to E^p$ and the map $\bar{\varphi}$ is given by
	\begin{align*}\left(\bar{\varphi}_x\right)_p\from K^{p+1}&\to G^{p}\times E^{p+1}\\
	    (k_0,\dots,k_p)&\mapsto (e,\dots,e,xk_0,\dots,xk_p).
	\end{align*}
	As $\|N(G^\bullet\times E)\|$ is contractible, the right side of the diagram induces homotopy equivalences in the geometric realization.
	
	A bisimplicial version of this construction is applied in \cite[Theorem 1.2]{FelisattiNeumann}: Marcello Felisatti and Frank Neumann work on the bisimplicial manifolds itself to construct a classifying map and prove a similar statement to \autoref{thm:unichernweil}. 
\end{remark}

\section{Vector bundles alias principal \texorpdfstring{$\Gl_n(\C)$}{Gln(C)}-bundles}
As (complex) vector bundles are of specific interest, we want to translate the statements, about equivariant characteristic forms in the last sections, from principal $\Gl_n(\C)$-bundles to their associated vector bundles. One can also replace $\Gl_n(\C)$ by subgroups, e.g., $U(n)$ to obtain analogues statements.

\begin{defn} Let $E$ be a principal $\Gl_n(\C)$-bundle. The associated vector bundle 
\[\mathcal E=E\times_{\Gl_n} \C^n\]
is the quotient of $E\times \C^n$ by the diagonal action of $\Gl_n$, where the action on $\C^n$ is given by matrix multiplication from the left.
\end{defn}
\begin{lemma}\label{lem:endiso}
    There is an isomorphism $C^\infty(E,\mathfrak{gl_n})^{\Gl_n}\to C^\infty(M,\End \mathcal E)$.
\end{lemma}
\begin{proof}
This map is well known, but we give a proof for completeness.
\[\End \mathcal E =\End (E\times_{\Gl_n} \C^n)=E\times_{\Gl_n} \End(\C^n)=E\times_{\Gl_n} M_n(\C)=E\times_{\Gl_n} \mathfrak{gl}_n\]
Hence, it suffices to construct the isomorphism \[C^\infty(E,\mathfrak{gl_n})^{\Gl_n}\to C^\infty(M,E\times_{\Gl_n} \mathfrak{gl}_n).\]
Therefore, let $f\in C^\infty(E,\mathfrak{gl_n})^{\Gl_n}$ and $s$ be a local section of $E\to M$. Then, 
\begin{align*}
    M\supset U&\to E\times_{\Gl_n} \mathfrak{gl}_n\\
m&\mapsto (s(m),f(s(m)))
\end{align*}
defines a local section of $\End\mathcal E\to M$. Picking another local section $s'$ around $m$, there exists $A\in\Gl_n$, s.t., $s'(m)=s(m)A$ and 
\begin{align*}(s'(m),f(s'(m)))&=(s(m)A,f(s(m)A))\\&=(s(m)A,A\invers f(s(m))A)\\&=(s(m),f(s(m)))\in E_m\times_{\Gl_n} \mathfrak{gl}_n.\end{align*}
Hence the image of $f$ is independent of the local section and defines an element in $ C^\infty(M,\End \mathcal E)$.

On the other hand given $f\in C^\infty(M,E\times_{\Gl_n} \mathfrak{gl}_n)$, then for $x\in E$, $f(\pi(x))=(xg,A_f(x))=(x,gA_f(x)g\invers)$ for some $A_f(x)\in \mathfrak{gl}_n$. The map $x\mapsto gA_f(x)g\invers$ is equivariant by definition and smooth as $\pi,f,$ and the action are smooth. We only have to check that both map are inverse to each other.

Start with $f\from E\to \mathfrak{gl}_n$. Let $x\in E$ and $s$ be a local section of $E$ around $\pi(x)$, then, for some $g\in\Gl_n$, 
\begin{multline}x\mapsto (s(\pi(x)),f(s(\pi(x))))=(xg,f(xg))\\=(x,gg\invers(f(x))gg\invers)=(x,f(x))\mapsto f(x).\end{multline}
Now let $f\in C^\infty(M,E\times_{\Gl_n} \mathfrak{gl}_n)$ and $s$ again be a local section of $E$. The composition is $m\mapsto (s(m),A_f(s(m)))=f(m)$, by the definition of $A_f$.
\end{proof}

A left $G$-action on $E$ which commutes with the action of $\Gl_n$ clearly induces a left $G$-action on $\mathcal E$.

There is a one to one correspondence between connections on the principal-$\Gl_n$-bundle and those one the associated vector bundle (see e.g. \cite[Ex. 3.4]{Baum}).

\begin{defn}[Def.\ 2.23.\ of \cite{Bunke}]\label{def:momentmap}
    Let $\nabla$ be a connection on the $G$-vector bundle $\mathcal E$. The \emphind{moment map} $\mu^\nabla\in\Hom(\mathfrak{g}, \omega^0(M, \End(\mathcal E)))^G$ \symindex[m]{\mu^\nabla} is defined by
\[\mu^\nabla(X) \wedge \varphi := \nabla_{X^\sharp_M} \varphi + L^{\mathcal E}_X \varphi,\quad \varphi \in\omega^0(M, \mathcal E).\]
Here $L^{\mathcal E}_X$ \symindex[l]{L^M_X} denotes the derivative
\[L^{\mathcal E}_X \varphi=\ddt\exp(tX)^*\varphi.\]
\end{defn}
\begin{remark}
	Observe that for a function $f\in C^\infty(M,\C)$, 
    \[L^{M\times \C}_X f(m)=\ddt\left(\exp(tX)^*f\right)(m)=\ddt f\left(\exp(tX)\invers m\right)=-df_m(X^\sharp_M).\]
Therefore, we altered the sign in the definition of \cite{Bunke}.
\end{remark}

\begin{theorem}\label{thm:compare}
    Let $\vartheta\in\Omega^1(E,\mathfrak{gl}_n)$ be a connection on the principal $\Gl_n$-bundle $E$ and $\nabla$ be the associated connection on the associated vector bundle $\mathcal E$. Then, with respect to the isomorphism of Lemma \ref{lem:endiso}, one has
\[d\vartheta+\vartheta\wedge \vartheta=R^\nabla \quad\text{ and }\quad \mu^\vartheta=\mu^\nabla.\]
\end{theorem}
\begin{proof}
The first statement is standard, while the second is stated in \cite[Lemma 3.2]{BerlineVergne83}, where the proof is left to the reader. Here it is: The statement is local. Let $s=(s_1,\dots,s_n)$ be a local frame, i.e., a local section of $E$, $X\in\mathfrak g$ and $\varphi\in C^\infty(M,\mathcal E)$. Locally one can write $\varphi=\sum_i \varphi_is_i$.
\begin{align*}
s^*(\mu^\vartheta(X))\varphi&=s^*(\iota(X^\sharp_E)\vartheta)\sum_i \varphi_is_i\\
&=\sum_i \varphi_i(s^*\vartheta[X^\sharp_E])s_i\\
&=\sum_i \varphi_i(s^*(\vartheta[ds\circ d\pi(X^\sharp_E)+(1-ds\circ d\pi)(X^\sharp_E)])s_i\\
&=\sum_i \varphi_i\left(s^*(\vartheta[ds (X^\sharp_M)] +\vartheta[(1-ds\circ d\pi)(X^\sharp_E)])\right)s_i\\
&\hspace{-0.6em}\overset{\text{see}}{\underset{\text{below}}{=}}\sum_i \varphi_i\left(s^*(\vartheta[ds(X^\sharp_M)])s_i+\ddt\exp(tX)\cdot(s_i)\right)\\
&=\sum_i \varphi_i\left(\nabla_{X^\sharp_M} s_i+\ddt\exp(tX)^*(s_i)\right)\\
&=\sum_i \big(d\varphi_i(X^\sharp_M)s_i+\varphi_i\nabla_{X^\sharp_M} s_i+\left(\ddt\exp(tX)^*\varphi_i\right)s_i\\
&\hspace{7em}+\varphi_i\ddt\exp(tX)^*(s_i)\big)\\
&=\nabla_{X^\sharp_M}\sum_i \varphi_is_i+\ddt\exp(tX)^*(\sum_i \varphi_is_i)\\
&=\nabla_{X^\sharp_M} \varphi + L^E_X \varphi\\
&=\mu^\nabla(X)\varphi
\end{align*}
For the step in the middle, we have to show that for any $m\in M$ the equation 
\begin{equation}\label{eq:stepinthemiddle}\vartheta_{s(m)}\left[(1-ds\circ d\pi)\left(\ddt \exp(tX) s(m)\right)\right]\cdot s=\ddt\exp(tX)(s(\exp(-tX)m))\end{equation}
holds, where the $\cdot$ on the left-hand side emphasizes that the action is from the right.
This is shown as follows. The vector field $(1-ds\circ d\pi)\left(\ddt \exp(tX) s(m)\right)$ is horizontal, since \[d\pi\left((1-ds\circ d\pi)\left(\ddt \exp(tX) s(m)\right)\right)=0.\] Hence there is a unique $Y\in \mathfrak{gl}_n$, such that
\[(1-ds\circ d\pi)\left(\ddt \exp(tX) s(m)\right)=s(m)Y\]
and thus $\vartheta_{s(m)}\left[(1-ds\circ d\pi)\left(\ddt \exp(tX) s(m)\right)\right]=Y$.

Now, calculate
\begin{align*}
(1-ds\circ d\pi)&\left(\ddt \exp(tX) (s(m))\right)\\
&=\ddt \exp(tX) s(m)-\ddt s\circ \pi(\exp(tX) s(m)),\\\intertext{hence, by equivariance of $\pi$,}
&=\ddt \left(\exp(-tX) s(m)-s(\exp(tX)\pi\circ s(m))\right)\\
&=\ddt \left(\exp(-tX) s(m)-s(\exp(tX)m)\right)\\
&=\ddt \exp(tX) s(\exp(-tX)m).
\end{align*}
Thus we have shown that shows that $Y=s(m)\invers\ddt \exp(tX) s(\exp(-tX)m)$ and hence equation \eqref{eq:stepinthemiddle} holds.
\end{proof}

Let $\vartheta,\vartheta'$ be connections on the principal $\Gl_n$-bundle $E$ and $\nabla,\nabla'$ the associated connections on the associated vector bundle $\mathcal E$. Then one has the transgression forms 
\[\widetilde{\omega}(\nabla,\nabla')=\int_I \omega(\nabla_t)=\widetilde{\omega}(\vartheta,\vartheta'),\]
which coincide in the sense of Lemma \ref{lem:endiso}.
Here $\nabla_t$ is, as above, the convex combination, which is a connection on $\Real\times \mathcal E\to \Real\times M$. \index{transgression form!vector bundle}\symindex[o]{\widetilde{\omega}(\nabla,\nabla')}

\section*{Acknowledgements}
The first author wants to thank the International Max Planck Research School {\it Mathematics in the Sciences} for financial support. This research was supported by ERC Starting Grant No.\ 277728.

\printbibliography
\end{document}